\documentclass[a4paper]{article}
\usepackage{amssymb}
\usepackage[margin=1.3in]{geometry}
\usepackage[affil-it]{authblk}
\usepackage{amsthm}
\usepackage{natbib}
\newtheorem{theorem}{Theorem}

\usepackage{amssymb}
\usepackage{amsmath}
\usepackage{pgfplots}
\usepackage{tikz}
\usepackage{pgfplots}
\usepackage{subcaption}
\pgfplotsset{compat=1.18}
\usepgfplotslibrary{groupplots}
\usepackage{amsthm}

\newtheorem{remark}{Remark}
\newtheorem{definition}{Definition}
\usepackage{algorithm}
\usepackage{algpseudocode}
\newtheorem{proposition}{Proposition}
\usepackage{caption}
\usepackage[table]{xcolor}
\usepackage{colortbl}
\usepackage{float}
\usepackage{placeins}
\usepackage{flafter} 
\usepackage{booktabs}
\usepackage{multirow}
\usepackage{threeparttable}
\usepackage{enumitem}
\usepackage{adjustbox}
\usepackage{hyperref}
\usepackage{setspace}
\usepackage{soul}
\usepackage{xcolor} 
\usepackage{fullpage}
\usepackage{subcaption}
\usepackage[normalem]{ulem}
\usepackage{orcidlink}

\usepackage{tikz}
\usetikzlibrary{arrows.meta}
\definecolor{kdblue}{HTML}{2140C0}
\definecolor{kdamber}{HTML}{EF9F27}
\definecolor{kdgreen}{HTML}{2E7D1E}
\definecolor{kdred}{HTML}{D6332B}
\tikzset{
  kdv/.style={circle,draw=black,fill=white,line width=0.5pt,minimum size=0.52cm,inner sep=0pt}, 
  kdd/.style={kdv,draw=kdblue,fill=kdblue,text=white},                                        
  kdband/.style={kdamber,opacity=0.30,line width=0.75cm,line cap=round},                      
  kdarr/.style={-{Stealth[length=5pt]},kdgreen,line width=1.1pt,shorten <=0.37cm,shorten >=0.4cm}
}
\newcommand{\kdx}{0.35} 
\newcommand{\kdpanel}[4]{\begin{tikzpicture}
  \useasboundingbox (-0.4-\kdx,-0.4) rectangle (2.6+\kdx,2.6);
  \coordinate (a) at (-\kdx,2.2);  \coordinate (e) at (2.2+\kdx,2.2);
  \coordinate (c) at (0.45,1.1);   \coordinate (d) at (1.75,1.1);
  \coordinate (b) at (-\kdx,0);    \coordinate (f) at (2.2+\kdx,0);
  \kdcover #1
  \draw[line width=0.5pt] (a)--(c) (b)--(c) (c)--(d) (d)--(e) (d)--(f);
  \foreach \v in {a,b,c,d,e,f} \node[kdv] at (\v) {$\v$};
  \foreach \v in {#2} \node[kdd] at (\v) {$\v$};
  \foreach \v in {#3} \draw[kdgreen,line width=0.9pt] (\v) circle (0.34cm);
  #4
\end{tikzpicture}}
\newcommand{\kdbandx}[3]{\def\kdp{#2}\def\kdq{#3}\ifx\kdp\kdq\fill[#1] (#2) circle (0.375cm);\else\draw[kdband,#1] (#2)--(#3);\fi}
\newcommand{\kdband}[2]{\kdbandx{kdamber,opacity=0.30}{#1}{#2}}
\newcommand{\kdgray}[2]{\kdbandx{gray,opacity=0.15}{#1}{#2}} 
\newcommand{\kdcover}{} 
\newcommand{\kdattacked}[1]{\draw[kdred,line width=0.7pt] (#1) circle (0.32cm) circle (0.37cm);}

\newcommand{\DSP}{$\textbf{DSP}(\hat{x})$\,}

\newcommand{\DSPA}{$\textbf{DSP}_A(\hat{x})$\,}
\makeatletter
\def\blfootnote{\xdef\@thefnmark{}\@footnotetext}
\makeatother

\begin{document}

\title{A Benders decomposition approach for the $k$-defensive domination problem\blfootnote{
		\textit{Email addresses:} 
        \texttt{bilge.varol@tau.edu.tr} (Bilge Varol),
		\texttt{tinaz.ekim@bogazici.edu.tr} (Tınaz Ekim), 
		\texttt{ktaninmis@ku.edu.tr} (Kübra Tanınmış)}
}

\author[1,2]{Bilge Varol \,\orcidlink{0000-0002-3252-4277}}
\author[1]{Tınaz Ekim \,\orcidlink{0000-0002-1171-9294}}
\author[3]{K\"ubra Tan{\i}nm{\i}\c{s} \,\orcidlink{0000-0003-1081-4182}}
		
\date{}

\affil[1]{ \small Department of Industrial Engineering, Boğaziçi University, Istanbul, Türkiye}

\affil[2]{\small Department of Industrial Engineering, Turkish-German University, Istanbul, Türkiye}

\affil[3]{\small Department of Industrial Engineering, Koç University, Istanbul, Türkiye}

\maketitle

\begin{abstract}
The $k$-defensive domination problem is a powerful modeling tool for strategic decision-making in network security and disaster/emergency management, where multiple nodes may be simultaneously under attack. 
Despite its practical relevance, the problem has been poorly studied, largely due to its high computational difficulty. 
This study investigates the application of Benders decomposition to the $k$-defensive domination problem, aiming to improve computational efficiency over standard integer programming formulations.  
Several cut generation strategies including a combinatorial approach and addition of multiple cuts simultaneously are proposed. Theoretical results on the strength of feasibility cuts are presented.
In addition, two novel enhancement strategies are proposed: a clique-cover-based heuristic, the first feasible solution method in the literature for this problem and an initial cut generation heuristic that in some cases resolves the problem without further branching. 
Experiments on Erd\H{o}s--R\'enyi, chordal, and Barab\'asi--Albert instances show that the algorithm variant involving all of the proposed components is able to solve instances that remain out of reach for classical formulations and standard Benders decomposition. Additionally, we solve real-world security network instances with up to 500 vertices and $k=10$.

\end{abstract}

\section{Introduction}\label{section:Introduction}
\paragraph{Motivation} Domination models are frequently used in security problems and disaster/emergency management applications
in networks. Assuming that a node where rescue/defense equipment is located has the ability to protect itself or one of its 
neighbors in the network, a dominating set represents the locations to place rescue/defense equipment so that we can meet
the need for emergency response that may arise at any node of the network. Since rescue/defense equipment is a costly resource, we seek a dominating set containing the least number of nodes. Due to its extensive applications and theoretical properties, the dominating set problem is widely studied in the literature (\cite{haynes2013fundamentals, henning2014total, haynes2020topics, haynes2023domination}). Several variations of the classical dominating set problem with high practical and theoretical relevance have also attracted a lot of attention (see e.g. \cite{haynes2013fundamentals, harary1995conditional}).

In this paper, we consider a relatively new and little-studied version of domination problems, called the
defensive domination or $k$-defensive domination problem ($k$-DefDom). When making strategic decisions in emergency response problems in networks,
this variant takes into account situations where security threats or the need for aid may occur at many places at the
same time and offers modeling capabilities that classical domination problems cannot provide. The
need for security/assistance can occur simultaneously at any $k > 1$ nodes that we do not know in advance, and
we need to be prepared to respond to all possible threats. A security (or assistance) team is only capable
of assisting one node. Therefore, a distinct team must be assigned to each of the nodes that are threatened at
the same time. 

A practical application of the defensive domination problem occurs in private security companies operating as an additional layer of urban protection, complementing public policing by offering paying clients (home or business) rapid alarm response (see \cite{coversurveyreport} for details). When an alarm is triggered, the operator must dispatch a patrol vehicle to assess the
situation within a strictly bounded time, on the order of 15 minutes, because competition in this market makes response speed a primary contractual differentiator. As such, the company should decide where to locate the patrol cars so that any alarm, eventually any subset of $k$ alarms occurring simultaneously, can be responded within 15 minutes. The company should decide jointly where to locate the patrol cars and how to dispatch them for $k$ simultaneous alerts. This problem can be modeled as a defensive domination problem in a \textit{security network} obtained by representing districts of a city by vertices and capturing the adjacency/vicinity between districts using edges (see Figure \ref{fig:securitynetwork} for an example). The districts are defined in such a way that the response time within a district is negligible and edges of the security network correspond to a 15 minutes drive for a patrol car. Besides, the idle cars are given the additional task of patrolling the districts to deter offenders and to sustain a reassuring presence. Accordingly, the company prefers to locate at most one patrol car at a district to allow a better satisfaction of patrolling constraints. In this setting, a $k$-defensive dominating set in the related security network provides a solution to the private security company's patrol car location and dispatching problem.

\begin{figure}[htbp]
    \centering
    \includegraphics[width=15cm]{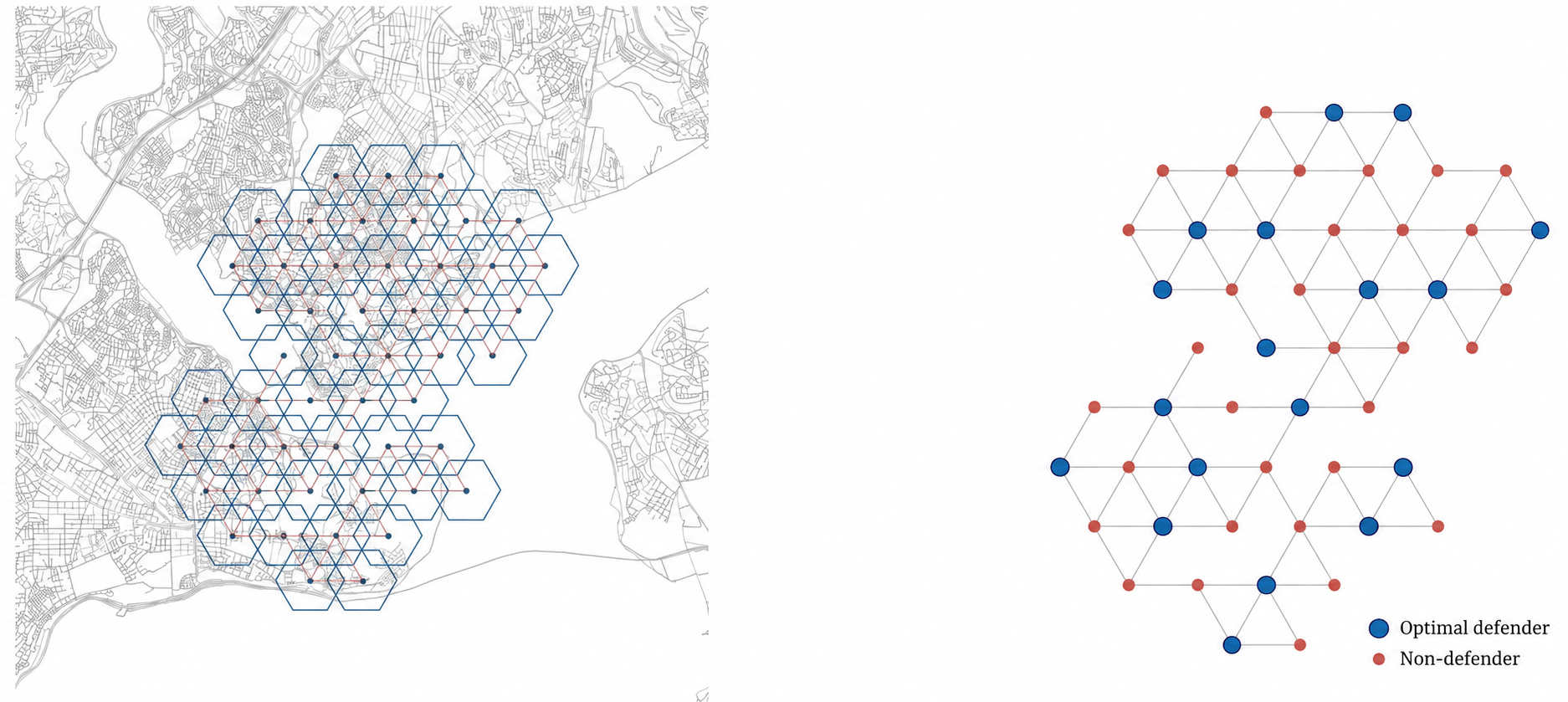}
    \caption{A security network with 50 nodes obtained from the map of Istanbul. Nodes represent hexagonal districts, and edges indicate feasible connections between neighboring districts. Some connections between neighboring districts are excluded because of geographical or accessibility constraints, such as seas, lakes, rivers, or mountainous terrain, which prevent a suitable road connection within 15 minutes. An optimal 2-defensive dominating set obtained in Section \ref{section:results} is illustrated by blue nodes in the security network.}
    \label{fig:securitynetwork}
\end{figure}

\paragraph{Related work} Introduced by \citep{farley2004defensive}, there are only a few studies on the defensive domination problem all of which focus on its computational complexity in special graph classes. \cite{ekim2020complexity} showed that the $k$-DefDom is $\mathcal{NP}$-complete even in split graphs when $k$ is fixed. Additionally, they demonstrated that deciding whether a given set of vertices forms a $k$-defensive dominating set is already co-$\mathcal{NP}$-hard, highlighting the problem’s inherent computational difficulty. They also showed that the problem can be solved in linear time in paths, cycles, co-chain graphs, and threshold graphs, even if $k$ is not fixed. \cite{ekim2023defensive} considered proper interval graphs and presented an $O(n + |B| \log k)$ time algorithm where $|B|$ is the number of bubbles in a bubble representation of a proper interval graph. 

\cite{dereniowski2019cops} introduced a variation of the $k$-DefDom, referred to as $A$-defensive domination problem. This problem takes as input a set of specific attack scenarios and supports assigning several defenders to a node. While the problem is shown to be $\mathcal{NP}$-complete even for split graphs, it is proved that $A$-defensive domination becomes solvable in polynomial time when restricted to interval graphs, due to the strong connection the authors establish between the defensive domination and the Cops and Robbers game with an infinitely fast robber. \cite{henning2021approximation,henning2025more} proved that the $k$-DefDom is $\mathcal{NP}$-complete for bipartite graphs with fixed $k>0$, whereas it is efficiently solvable for complete bipartite graphs. They also investigated the approximability properties of the defensive domination problem and showed that it is APX-complete for bounded-degree graphs. 

Recently, \cite{chaplick2025note} settled the complexity of the defensive domination problem by proving that it is $\Sigma_2^P$-complete, placing it in the second level of the polynomial hierarchy. This result provides a formal explanation for the inherent difficulty of the problem and suggests that defensive domination is unlikely to admit polynomial-time algorithms or compact exact formulations in the general case. The authors further considered a natural multi-set variant of defensive domination where each defender has a multiplicity representing its ability to respond to multiple attacks simultaneously. In practice, a defensive dominating multi-set corresponds to assigning capacities (several emergency teams) to defenders. Although this extension makes the problem closer to classic capacitated facility location problems, the multi-set defensive domination problem differentiates from them in important ways which rules out a direct transfer of related solution methods. The aim is not minimizing the total transportation costs between facilities and demand points, instead it is required that a distinct emergency team is located at distance at most one for each node of every potential subset of $k$ demand points in the network. In \cite{chaplick2025note}, the authors demonstrated that the multi-set defensive dominating set is solvable in polynomial time when restricted to interval graphs. However, they noted that the complexity status of the original defensive domination problem on general interval graphs remains open. Lastly, \cite{continuousdefdom} investigate the complexity of the defensive $\delta$-covering problem which generalizes the defensive domination problem by relaxing the defense condition as follows. A defense counters an attack if there is a matching of the vertices in
the defense to the vertices in the attack, such that every vertex in the attack is matched to a distinct defender which has distance at most $\delta$. The authors show that this generalization is $\Sigma_2^P$-complete for any $\delta\geq 1$ and analyze how the complexity changes if the
attacks or defenses may be a multi-set. If the defense is allowed to be a multi-set (capacitated defenders), the complexity does
not change in any case they consider. However, if the attack is allowed to be a multi-set, corresponding to more severe attacks requiring multiple defenders to be responded to, the problem often becomes easier.

\paragraph{Our contribution} To date, there has been no exact solution approach developed for the $k$-DefDom in general graphs. Even though the theoretical complexity of the defensive domination and its various extensions have been studied, their computational aspects have not been investigated from an optimization point of view. Motivated by this gap, in this work, we develop a Benders Decomposition based exact method to solve the problem in general graphs. Our contribution and the outline of the paper can be summarized as follows.

A novel integer programming (IP) formulation is introduced for the $k$-DefDom in Section \ref{section: ProblemDefinition}. In Section \ref{section:Benders}, we introduce a Benders decomposition framework to address the challenge posed by a large number of possible attacks, by exploiting the special structure of the problem, which simplifies to solving multiple bipartite matching problems once the defender choices are fixed.
Furthermore, a combinatorial separation method that avoids solving the dual subproblems using Linear Programming (LP) is presented. Instead of standard LP-based separation, a recursive algorithm is employed. This algorithm searches for attack scenarios that are not countered by the current defenders by examining the graph structure directly. Consequently, violations are identified and feasibility cuts are generated without the computational cost of solving many LPs. 

To further improve the performance of the proposed decomposition algorithm, we introduce two heuristic approaches in Section \ref{section:heuristics}. The first, an initial cut generation heuristic, focuses on generating strong initial feasibility cuts. The second, a clique-cover-based heuristic, generates high-quality feasible solutions and thus tight initial upper bounds, which support an efficient pruning of the search space. In addition to enhancing the performance of our exact methods, this latter heuristic is, to the best of our knowledge, the first to guarantee a feasible defensive dominating set for any input graph, and therefore constitutes a contribution in its own right. In Section \ref{section:combinatorialCutGeneration}, we further discuss strategies to generate cuts and solve the subproblems, each one yielding a variant of our approach. In Section \ref{section:results}, we evaluate the performance of all methods proposed in this paper under different implementation choices, including several cut sampling strategies, using a comprehensive experimental test-bed based on Erdös-Rényi graphs, Barabási-Albert graphs, chordal graphs, and real-world security network instances.
We discuss the limitations and future research directions in Section \ref{section:conclusion}.
 
\section{Problem definition and integer programming formulation}\label{section: ProblemDefinition}

Let $G=(V,E)$ be a graph with the set of vertices $V$ and the set of undirected edges $E$, where each edge $(u,v)\in E$ connects a pair of distinct vertices $u$ and $v$. 
The classical domination problem seeks to find a subset of vertices $D \subseteq V$, which represent defender vertices of $G$ and the goal is to ensure that every vertex in $V$ is either in $D$ or is adjacent to a vertex in $D$. The $k$-DefDom generalizes this idea by incorporating multiple and simultaneous attack scenarios, which are called \textit{$k$-attacks}. In this section, we first provide a formal definition of the problem and then, we introduce a mixed-integer programming (MIP) formulation based on the explicit assignment of defenders to attacked vertices for each attack scenario. 

Formally, a $k$-attack on $G$ is a set of $k$ distinct vertices $\{a_1,...,a_k\}$ that are under attack. A $k$-attack $A$ can be \textit{countered} (or \textit{defended}) by a subset \(X \subseteq V\) of vertices, called \textit{defenders}, if and only if there exists an injective function $f: A \to X$, such that for each $i\in A$ either $f(a_i) = a_i$ or $(a_i, f(a_i)) \in E$. This implies that there exists an assignment of $k$ distinct members of $X$ to the members of $A$, represented as ${(a_1,x_1),...,(a_k,x_k)}$, such that either $a_i =x_i$ or $(a_i,x_i)$ is an edge of $G$, for all $i, 1 \leq i \leq k$ \citep{ekim2020complexity}. 
A subset $D\subseteq V$ is a \emph{$k$-defensive dominating set} of $G$ if and only if $D$ can counter any $k$-attack in $G$. Figure~\ref{fig:kdd_example} illustrates the attack--defense mechanism of the $k$-DefDom through a simple example. It is followed by the definition of the \emph{$k$-DefDom}.

\begin{figure}[ht]
\centering
\begin{subfigure}[t]{0.3\linewidth}\centering
  \kdpanel{}{a,c,d,e}{}{\kdattacked{c}\kdattacked{f}
    \draw[kdarr,shorten <=0pt,shorten >=0pt] ([shift=(25:0.5cm)]c) arc (25:110:0.5cm);
    \draw[kdarr] (d) to (f);}
  \caption{$D_1=\{a,c,d,e\}$, $A=\{c,f\}$: a valid assignment of defenders.}
\end{subfigure}\hspace{0.02\linewidth}
\begin{subfigure}[t]{0.3\linewidth}\centering
  \kdpanel{}{c,d}{}{\kdattacked{a}\kdattacked{b}
    \draw[kdarr,bend right=15] (c) to (a);
    \draw[kdred,line width=1.3pt] (-0.35,0.55)--(-0.15,0.75) (-0.35,0.75)--(-0.15,0.55);}
  \caption{$D_2=\{c,d\}$, $A=\{a,b\}$: $A$ cannot be countered.}
\end{subfigure}

\vspace{0.8em}
\small
\tikz[baseline=-0.6ex]\node[kdd]{}; defender ($\in D$)\quad
\tikz[baseline=-0.6ex]\node[kdv]{}; non-defender ($\notin D$)\quad
\tikz[baseline=-0.6ex]{\node[kdv](n){};\kdattacked{n}} attacked vertex\quad
\tikz[baseline=-0.6ex]\draw[kdarr,shorten <=0pt,shorten >=0pt](0,0)--(0.8,0); defender assignment
\caption{Illustration of $k$-defensive domination on a double star graph for $k=2$. The attacked vertices must be assigned distinct defenders: in (a) $f$ can only be defended by $d$, whereas $c$ can be defended by $a$ or by itself. Any $2$-attack can be defended similarly, so $D_1$ is a $2$-defensive dominating set (and an optimal one). In (b), only $c$ is adjacent to the attacked vertices, so $b$ remains undefended. $D_2$ is not a $2$-defensive dominating set.}
 \label{fig:kdd_example}
\end{figure}

\begin{definition}[\textbf{$k$-DefDom}]
Given $G=(V,E)$, find $D\subseteq V$ of minimum cardinality such that $D$ is a $k$-defensive dominating set of $G$.
\end{definition}

Theorem \ref{thm:coNPhard} emphasizes that the $k$-DefDom is significantly more difficult than other domination problems by showing that it is not even in NP. This result points out in particular the difficulty of obtaining a feasible defensive dominating set by means of an efficient heuristic algorithm. 

\begin{theorem}[\cite{ekim2020complexity}]
\label{thm:coNPhard}
Given a graph $G$, an integer $k$, and a set $D \subseteq V(G)$, it is co-NP-hard to decide whether $D$ is a $k$-defensive dominating set.
\end{theorem}

Next, we develop an integer programming formulation for the $k$-DefDom. Let $\mathcal{A}$ represent the set of all possible $k$-attack combinations in $G$, where each attack scenario $A \in \mathcal{A}$ corresponds to a subset of $k$ targeted vertices $A \subseteq V$. The closed neighborhood of a vertex $v\in V$, denoted as $N[v]$, is the set of all vertices adjacent to $v$ including itself. In other words, $N[v]=N(v)\cup \{v\}$ where $N(v)$ is the open neighborhood of $v$. The closed neighborhood of any set $S\subseteq V$ is defined as $N[S]=\cup_{v\in S} N[v]$, i.e., the set of vertices that are in $S$, or adjacent to at least one vertex in $S$.
The decision variables and the integer programming formulation of the problem are presented below. 
\begin{center}
\small
$\begin{array}{@{} l @{\hspace{1cm}} l @{}}
x_i = \begin{cases} 
1 & \text{if } i \text{ is in the defensive dominating set,} \\
0 & \text{otherwise.}
\end{cases} 
&
y_{ij}^{A} = \begin{cases} 
1 & \parbox[t]{7cm}{if $i\in V$ defends $j \in A$ in attack scenario $A\in \mathcal{A}$,} \\
0 & \text{otherwise.}
\end{cases}
\end{array}$
\end{center}
\vspace{-0.4cm}

\begin{align}
\textbf{(kDD):}\quad &\text{Minimize} \quad \sum_{i \in V} x_i \label{model:3.1} \\
\text{subject to: }&\sum_{i \in N[j]} y_{ij}^{A} = 1 \quad \forall A \in \mathcal{A}, j \in A \label{model:3.2} \\
&\sum_{j \in A \cap N[i]} y_{ij}^{A} \leq x_i \quad \forall A \in \mathcal{A}, i \in N[A] \label{model:3.3} \\
&x_i \in \{0, 1\} \quad \forall i \in V \label{model:3.4} \\
&y_{ij}^{A} \in \{0, 1\} \quad \forall A \in \mathcal{A}, j \in A, i \in N[j]  \label{model:3.5} 
\end{align}

The aim of this problem is to find the smallest possible $k$-defensive dominating set in graph $G = (V, E)$. The objective function \eqref{model:3.1} minimizes the size of the defensive dominating set, i.e., the number of defender vertices. Constraint \eqref{model:3.2} ensures that each attacked vertex  $j$  in the attack set  $A$  is defended by exactly one vertex from its closed neighborhood  $N[j]$. Note that the index $A$ in variables $y_{ij}^{A}$ is needed to allow a defender $i$ to defend different vertices under different attack scenarios. Constraint \eqref{model:3.3} links the defending variable $y_{ij}^A$ to the membership variable  $x_i$, ensuring that a vertex $i$ can defend others only if it is part of the dominating set, and $i$ can defend at most one vertex in each attack scenario. Finally, constraints \eqref{model:3.4} and \eqref{model:3.5} enforce the binary nature of the decision variables. This formulation ensures a defensive strategy that minimizes the size of the dominating set while guaranteeing effective defense against all possible $k$-attacks.

In the model (kDD), it is observed that the decision variable $y_{ij}^{A}$ can be relaxed as continuous without affecting the solution's integrality. When $x_i$ values are fixed, the problem reduces to determining $y_{ij}^{A}$, which represents attack scenario-based defender-attacker assignments. This allows the problem to be transformed into a maximum matching problem in bipartite graphs, where one set of vertices corresponds to the defenders and the other set to the attacked vertices. For each attack scenario $A_l \in \mathcal{A}$, where $l = 1, \dots, |\mathcal{A}|$, a bipartite graph $ G_l = (D \cup A_l, E_l) $ is constructed, with edges $ \{i,j\} $ such that $i \in D $ and $j \in (N[i] \cap A_l)$. It is well known that the linear programming (LP) relaxation of the maximum matching problem in bipartite graphs always yields an integer solution. Given that the subproblem formulation with fixed $x_i$ values can be interpreted as a maximum matching problem across multiple (as many as $|\mathcal{A}|$) bipartite graphs, the solution space remains integer-valued whenever the right-hand side coefficients are integers. As a result, in the optimal solution, the variables $y_{ij}^{A}$ naturally take binary (0–1) values, even when only non-negativity and upper bound constraints are imposed. This means that solving the relaxed version of the model still produces integer solutions due to the integrality of extreme points in the polyhedron defined by the feasible region.

\section{Benders decomposition approach}\label{section:Benders}

In this section, we present a Benders decomposition reformulation of (kDD) by exploiting its special structure. Following the definitions of the master problem and the subproblems, we present results on the structure of the subproblems’ feasible region and propose a combinatorial approach for generating the Benders cuts based on this analysis, which can substantially reduce the computational effort required by the method.

\subsection{A Benders decomposition model}
\label{section:BasicBenders}

Benders decomposition \citep{benders1962partitioning} is a widely used decomposition method to solve large-scale mixed-integer programming (MIP) and other complex optimization problems efficiently, by leveraging a particular block structure. This structure allows the problem to be split into smaller, substantially easier subproblems. In Benders decomposition, the variables are split into two sets such that the master problem is solved over one set, and the subproblem which is parametrized by the master problem's solution is solved over the remaining variables. While solving these problems iteratively, the information obtained from the subproblem is translated into feasibility or optimality cuts and fed into the master problem. For more detailed information on the Benders decomposition and its implementation, see \cite{taskin2010benders}.

One class of optimization problems for which the Benders decomposition method is well suited is two-stage stochastic programs, due to the inherent partition of its decision variables into first-stage and second-stage variables. The formulation (kDD) we develop for the $k$-DefDom possesses a similar structure, where the selection of defenders can be seen as first-stage variables, and the assignment of the defenders to the nodes under attack can be seen as the second-stage variables. This makes Benders decomposition a promising candidate as a solution approach. In the following, we present how to apply the decomposition and further exploit the problem structure to design an efficient exact algorithm for the $k$-defensive domination problem.

Using the maximum-matching results presented in Section \ref{section: ProblemDefinition}, in this section we adopt the relaxed (kDD) model with continuous $y_{ij}^{A}$ variables. In the logic of Benders decomposition approach, the master problem (MP) contains the complicating variables, while the subproblem (SP) contains the remaining variables. In this problem, complicating variables are $x_i$ because of their integral nature. Then, the master problem will only have $x_i$ variables along with the feasibility constraints. The optimal solution of the relaxed MP, which is denoted by $\hat{x}_i$, will be fed into the subproblem as a parameter. Then, the subproblem will be solved with fixed $\hat{x}_i$ values to obtain $\hat{y}_{ij}^{A}$. 

\begin{align}
\textbf{MP:}\quad 
& \text{Minimize} \quad \sum_{i \in V} x_i 
\label{model:4.6} \\
\text{subject to:}\quad 
& \text{$x \in  \mathcal{D} $} 
\label{model:4.7}\\
& x_i \in \{0, 1\} \quad \forall i \in V 
\label{model:4.8} \\
\textbf{SP}(\hat{x}):\quad 
& \text{Minimize} \quad 0 
\label{model:4.9} \\
\text{subject to:}\quad 
& \sum_{i \in N[j]} y_{ij}^{A} = 1 
\quad \forall A \in \mathcal{A}, j \in A
\label{model:4.10} \\
& \sum_{j \in A\cap N[i]} y_{ij}^{A} \leq \hat{x}_i 
\quad \forall A \in \mathcal{A}, i \in N[A] 
\label{model:4.11} \\
& y_{ij}^{A} \geq 0 
\quad \forall A \in \mathcal{A}, j \in A, i \in N[j]  
\label{model:4.12}
\end{align}

The MP is a pure IP model equivalent to (kDD), that seeks to find an optimal defender set that minimizes the total number of selected defenders in the objective function \eqref{model:4.6}. Constraint \eqref{model:4.7} is the feasibility condition, requiring that $x$ belongs to $\mathcal{D}$, the set of all $k$-defensive dominating sets of $G$.
 
In the Benders decomposition method, a \textit{relaxed} MP omitting \eqref{model:4.7} and involving a subset of Benders feasibility cuts is solved in each iteration.
Constraint \eqref{model:4.8} presents the binary decision variable $x_i$. 

The SP($\hat{x}$) is an LP that determines the feasibility of the given defender set $\hat{x}$ obtained via solving a relaxed MP, by checking whether it can successfully defend against all possible attack scenarios. 
Although SP($\hat{x}$) is an LP, it is non-compact due to the exponential cardinality of $\mathcal{A}$ with respect to $k$, which aligns with the co-NP-hardness result in Theorem \ref{thm:coNPhard}.
The objective function \eqref{model:4.9} trivially minimizes zero, 
since the objective in the main model involves $x_i$ variables only. The constraint sets \eqref{model:4.10}, \eqref{model:4.11} are the original constraints that are included in the main model. Finally, constraint set \eqref{model:4.12} presents the $y_{ij}^{A}$ variables, which are continuous. 
If $\hat{x}$ represents a feasible defender set, then SP($\hat{x}$) is feasible, otherwise it is infeasible. For a given $\hat{x}$, if SP($\hat{x}$) has a feasible solution $\hat{y}$, then $(\hat{x}, \hat{y})$ corresponds to a feasible solution of the original problem. However, if SP($\hat{x}$) does not yield a feasible solution, it means that $\hat{x}$ does not correspond to a $k$-defensive dominating set and the MP must be adjusted so that $\hat{x}$ is not considered in future iterations. This is done by adding feasibility cuts by using the theory of LP duality. The dual formulation of SP($\hat{x}$) is presented in equations \eqref{model:4.13}-\eqref{model:4.16}.
\begin{align}
\textbf{DSP($\hat{x}$): }&\text{Maximize} \quad \sum_{A \in \mathcal{A}} \sum_{j \in A} \alpha^A_j + \sum_{A \in \mathcal{A}} \sum_{i \in N[A]} \beta^A_i \hat{x}_i \label{model:4.13} \\
\text{subject to: }&\quad \alpha^A_j + \beta^A_i \leq 0, \quad \forall A \in \mathcal{A}, j \in A, i \in N[j] \label{model:4.14}\\
&\quad \alpha^A_j \in \mathbb{R} \quad \forall A \in \mathcal{A}, j \in A  \label{model:4.15}\\
&\quad \beta^A_i \leq 0, \quad \forall A \in \mathcal{A}, i \in N[A] \label{model:4.16}
\end{align}

If SP($\hat{x}$) is infeasible, then DSP($\hat{x}$) is unbounded, as it is always feasible (the null vector is a feasible solution). In this case, there exist an extreme ray $(\hat{\alpha},\hat{\beta})$ of DSP($\hat{x}$) such that $\sum_{A \in \mathcal{A}} \sum_{j \in A} \hat{\alpha}^A_j + \sum_{A \in \mathcal{A}} \sum_{i \in N[A]} \hat{\beta}^A_i x_i > 0$. We add the following \textit{feasibility cut} to the relaxed MP and solve it again.
\begin{equation}
    \sum_{A \in \mathcal{A}} \sum_{j \in A} \hat{\alpha}^A_j + \sum_{A \in \mathcal{A}} \sum_{i \in N[A]} \hat{\beta}^A_i x_i \leq 0
\end{equation}

By iteratively solving the MP and the SP, Benders decomposition refines the defender selection by adding feasibility cuts when the subproblem is infeasible, and stops when there are no violated feasibility cuts. In this setting, optimality cuts are not needed since the subproblem SP($\hat{x}$) is only a feasibility problem for a given relaxed MP solution $\hat{x}$.

We note that, the subproblem solves the model for all attack scenarios, but this is not practical since the main complexity comes from the number of attack scenarios. Since in the subproblem each scenario $A \in \mathcal{A}$ is independent of the others, we can exploit this structure to decompose it into multiple subproblems, solving one problem per attack scenario where we examine the existence of a matching of size $k$ between the current attack and the defenders presented by the vector $\hat{x}$. 
%
%
The resulting disaggregated dual subproblem is formulated as follows:
\begin{align}
\textbf{DSP}_A(\hat{x}): &\text{Maximize} \quad  \sum_{j \in A} \alpha_j +  \sum_{i \in N[A]} \beta_i \hat{x}_i \label{model:4.21} \\
\text{subject to: }&\quad \alpha_j + \beta_i \leq 0, \quad \forall j \in A, i \in N[j]\label{model:4.22}\\
&\quad \alpha_j \in \mathbb{R} \quad \forall j \in A \label{model:4.23}\\
&\quad \beta_i \leq 0, \quad \forall i \in N[A] \label{model:4.24}
\end{align}

Decomposing the subproblem into independent smaller subproblems allows us to apply different solution techniques. Generating one feasibility cut per scenario instead of a single aggregated cut might result in stronger cuts and speed up the convergence. A potential drawback is that the number of cuts can grow very large, which may lead to longer MP solution times. Identifying the most effective cut generation strategy such as cut sampling will therefore be a central focus of our experimental study, in which we will also examine the practical implications of the findings presented in the next section.

\subsection{A combinatorial approach for cut separation}
\label{section:combinatorialCuts}
In this section, we present a method to obtain Benders cut using a combinatorial procedure, instead of solving dual subproblem \DSPA as an LP. It is easy to observe that the null vector $(\alpha, \beta)=(0,0)$ is feasible for any attack $A\in \mathcal{A}$. For an improved solution, some $\alpha_j$ must be greater than zero. Additionally, as the value of a particular $\alpha_j$ increases, the value of each $\beta_i, i\in N[j]$ decreases by the same amount. This decrease is reflected to the objective value only for $i\in N[j]$ such that $\hat x_i=1$, assuming $\hat{x}\in \{0,1\}$. In Theorem \ref{thm:combinatorialCut}, we formally describe an approach to use this observation for obtaining a violated feasibility cut, i.e., an improving unboundedness direction of \DSPA, if any exists.

\begin{theorem}
\label{thm:combinatorialCut}
    Let $\hat{x}\in \{0,1\}^{|V|}$ be a given defender configuration and $A\in \mathcal{A}$ be an attack scenario. Let $S\subseteq A$ be a partial attack such that $|\{i\in N[S]: \hat{x}_i=1\}| < |S|$.
Additionally, let $(\hat \alpha,\hat \beta) \in \mathbb{R}^{|A|+|N[A]|}$ meet the following conditions:
\begin{enumerate}
    \item $\hat \alpha_j=1$ for all $j\in S$, and $\hat \alpha_j=0$ for all $j \in A \setminus S$.
    \item $\hat \beta _i=-1$ for all $i\in N[S]$, and $\hat \beta _i=0$ for all $i \in N[A] \setminus N[S]$.
\end{enumerate}
$(\hat \alpha,\hat \beta)$ is an improving ray of \DSPA. It is also extreme if $N[j] \not \subseteq N[S]$ for all $j\in A \setminus S$.

\end{theorem}
\begin{proof}
    Let $(\hat \alpha,\hat \beta)$ be a vector satisfying 1 and 2. 
    We need to show that $(\hat \alpha,\hat \beta)$ is a feasible direction, the objective function value increases in the direction of $(\hat \alpha,\hat \beta)$, and it is an extreme direction if the condition $N[j] \not \subseteq N[S]$ for all $j\in A \setminus S$ is satisfied. 
    
    \textbf{Feasibility:} To show that $(\hat \alpha,\hat \beta)$ is a direction (and thus it defines a ray at vertex $(0,0)$), we need to have $\hat{\beta}\leq 0$ and $\hat\alpha_j + \hat \beta_i \leq 0$, for all pairs $j \in A$ and  $i \in N[j]$. The former is clearly satisfied. To show the latter, consider all possible (distinct cases) for $i$ and $j$:

    Case (i): $j\in S$. In this case, $\hat{\alpha}_j=1$ by condition 1. Since $i\in N[j]$ and $N[j]\subseteq N[S]$, $i\in  N[S]$ holds. $\hat \beta_i=-1$ by condition 2. So,
    $\hat \alpha_j + \hat \beta_i =0  \leq 0$.
    
    Case (ii): $j\in A \setminus S$ and $i \in N[j]\cap N[S]$. In this case, $\hat{\alpha}_j=0$ by condition 1 and    $\hat \beta_i=-1$ by condition 2. So, $\hat \alpha_j + \hat{\beta}_i =-1 \leq 0$.

    Case (iii): $j\in A \setminus S$ and $i \in N[j] \setminus N[S]$. In this case, $\hat{\alpha}_j=0$ by condition 1 and    $\hat \beta_i=0$ by condition 2. So, $\hat \alpha_j + \hat{\beta}_i =0 \leq 0$.
    
    \textbf{Improvement:} We need to show that $\sum_{j \in A} \hat \alpha_j +  \sum_{i \in N[A]} \hat \beta_i \hat{x}_i > 0$. 
    By definition of $(\hat \alpha,\hat \beta)$, we have $\sum_{j \in A} \hat \alpha_j=\sum_{j \in S} \hat \alpha_j=|S|$. Similarly, $\sum_{i \in N[A]} \hat \beta_i \hat{x}_i= \sum_{i \in N[S]} \hat \beta_i \hat{x}_i=\sum_{i \in N[S]:\hat x_i=1} \hat \beta_i = -|\{i \in N[S]:\hat x_i=1\}|$. Thus, $\sum_{j \in A} \hat \alpha_j +  \sum_{i \in N[A]} \hat \beta_i \hat{x}_i= |S|-|\{i\in N[S]: \hat{x}_i=1\}|$ which is strictly positive under the current choice of $S$.

\textbf{Extremity:} Lastly, we will show that $(\hat \alpha,\hat \beta)$ constitutes an extreme direction when $N[j] \not \subseteq N[S]$ for all $j\in A \setminus S$. 
     Consider the $L_\infty$ normalization constraint: $\|x\|_{\infty} = \max_{1 \leq i \leq n} |x_i|\leq 1$. \DSPA becomes bounded under the normalization constraint and its extreme points are extreme directions of the original \DSPA. 
     Assume for a contradiction that $(\hat \alpha,\hat \beta)$ is not an extreme point. Then, it could be written as a convex combination of two distinct points. Let $(\hat \alpha^1,\hat \beta^1)$ and $(\hat \alpha^2,\hat \beta^2)$ be such points.
     $(\hat \alpha,\hat \beta)=\lambda (\hat \alpha^1,\hat \beta^1)+(1-\lambda )(\hat \alpha^2,\hat \beta^2)$ for some $\lambda\in(0,1)$. Since, $\hat \alpha_j=\lambda \alpha^1_j+(1-\lambda)\alpha^2_j=1$ for each $j \in S$, and $\alpha\leq 1$ by normalization, $\alpha^1_j=\alpha^2_j=1$. 
     Similarly, $\beta^1_i=\beta^2_i=-1$ for all $i\in N[S]$ and $\beta^1_i=\beta^2_i=0$ for all $i\in N[A] \setminus N[S]$ since $\beta_i\in [-1,0]$ after normalization. 
    For $(\hat \alpha^1,\hat \beta^1)$ and $(\hat \alpha^2,\hat \beta^2)$ to be distinct, there must exist some $j\in A \setminus S$ such that $\alpha^1_j\not=\alpha^2_j$ and $\hat \alpha_j=0=\lambda \alpha^1_j +(1-\lambda)\alpha^2_j$. This implies, w.l.o.g., $\alpha^1_j>0$, which violates the constraint $\alpha^1_j+\beta^1_i\leq 0$ for some $i\in N[j] \setminus N[S]$ (under our initial assumption that $N[j] \setminus N[S]$ is not empty). It follows that $(\hat \alpha,\hat \beta)$ is an extreme point, and thus an extreme ray of the original \DSPA.
\end{proof}

Next, we explain the connection between the extreme rays described in Theorem \ref{thm:combinatorialCut} and Hall’s Theorem which can be stated as ``$D$ is $k$-defensive if and only if $|N[S] \cap D| \geq |S|$ for every set $S$ with at most $k$ vertices''. We first note that a \textit{Hall violator} is a set $S$ violating Hall's Theorem for a given defender set $D$, i.e., $|S|\leq k,|N[S] \cap D| < |S|$.

\begin{remark}
    For a given attack $A$, a subset $S\subseteq A$ satisfies the cardinality condition in Theorem \ref{thm:combinatorialCut} if and only if it is a Hall violator.
    As a result, every Hall violator $S$ for a given $\hat x$, satisfying that $N[j] \not \subseteq N[S]$ for all $j\in A \setminus S$, gives an improving extreme ray, thus a violated feasibility cut.
    When the proposed values of $\hat \alpha$ and $\hat\beta$ are substituted into $\sum_{j \in A} \hat \alpha_j +  \sum_{i \in N[A]} \hat \beta_i \hat{x}_i\leq 0$, the Benders feasibility cut becomes $\sum_{i\in N[A]} x_i\geq |S|$. 
\end{remark}

If there does not exist $S\subseteq A$ such that $|\{i\in N[S]: \hat{x}_i=1\}| < |S|\}$ then, there does not exist a Hall violator for the given defender configuration $\hat x$, so $\hat x$ can defend $A\in  \mathcal{A}$. In this case, the subproblem is feasible, and there are no violated Benders feasibility cuts for $A$. The solution to \DSPA described in Theorem \ref{thm:combinatorialCut} still provides a direction (because the feasible region of \DSPA remains unchanged), but the objective value of \DSPA, which is a function of $\hat{x}$, does not improve in that direction.

It is easy to observe that, while determining the dual solution in the way proposed in Theorem \ref{thm:combinatorialCut} we could also set $\beta_i=-1$ for $i\in N[A] \setminus N[S]$ given that $\hat{x}_i=0$ (instead of $\beta_i=0$). The solution would remain feasible and the amount of violation at $\hat x$ would not change. However, the cut with $\beta_i=0$ for $i\in N[A] \setminus N[S]$ is tighter, as the constant part of the cut remains the same. 
Among alternative sets inducing a violated Benders feasibility cut (Hall violators), there exist situations in which one such set yields a strictly stronger cut that dominates the cut generated by another set. In the following proposition, we show that, under appropriate conditions, choosing a larger violator set yields a stronger feasibility cut.

\begin{proposition}
\label{prop:strongerCut}
   Let $S\subseteq A$ and $S'\subseteq A$ be Hall violators such that $S\subset S'$. If $|S'\setminus S|\geq |N[S'] \setminus N[S]|$, then the Benders cut that $S'$ defines according to Theorem 1 is at least as strong as the cut that $S$ defines.
\end{proposition}
\begin{proof}
    Let $C$ and $C'$ be the cuts associated with the sets $S$ and $S'$, respectively (i.e., $C:  \sum_{i\in N[S]} x_i\geq |S|$ and $C':  \sum_{i\in N[S']} x_i\geq |S'|$).
%
    We will show that any solution $x \in \{0,1\}^{|V|}$ satisfying $C'$ satisfies $C$. 
    Assume that $x$ satisfies the inequality $C'$. Since $S \subset S'$, it follows that $N[S] \subseteq N[S']$. We can partition the summation over $N[S']$ into two disjoint parts: 
    $$\sum_{i \in N[S']} x_i = \sum_{i \in N[S]} x_i + \sum_{i \in N[S'] \setminus N[S]} x_i \geq |S'|$$
Rearranging the inequality to isolate the sum associated with $C$:
$$\sum_{i \in N[S]} x_i \geq |S'| - \sum_{i \in N[S'] \setminus N[S]} x_i$$
Since $x_i \in \{0,1\}$, the maximum possible value of the summation $\sum_{i \in N[S'] \setminus N[S]} x_i$ corresponds to the cardinality of the set $N[S'] \setminus N[S]$ . Therefore, we have the following lower bound:
$$\sum_{i \in N[S]} x_i \geq |S'| - |N[S'] \setminus N[S]| = |S| + |S' \setminus S| - |N[S'] \setminus N[S]|$$
The equality is obtained by substituting $|S'| = |S| + |S' \setminus S|$ into the first inequality, using $S\subset S'$.
Then, using the proposition's hypothesis that $|S' \setminus S| \geq |N[S'] \setminus N[S]|$, the term $(|S' \setminus S| - |N[S'] \setminus N[S]|)$ is non-negative. It follows that $\sum_{i \in N[S]} x_i \geq |S|$, which is the inequality defining cut $C$. Therefore, $C'$ implies $C$, making $C'$ the stronger cut.
\end{proof}

During the cut-generation procedure described in 
Section \ref{section:combinatorialCutGeneration} we incorporate the cut-strength result established in Proposition \ref{prop:strongerCut} to guide the sampling of violated Benders feasibility cuts associated with a given infeasible defender configuration.

\section{Heuristic methods}
\label{section:heuristics}

In the Benders decomposition framework that we propose, the objective value of the MP yields a lower bound on the true optimal objective value of (kDD), 
as the MP outputs solutions to a relaxed problem until it finds a feasible solution, which is then guaranteed to be optimal under the classical iterative implementation of Benders decomposition. The initial lower bound is zero, because the MP without any constraints would find the null vector as the optimal solution. On the other hand, an upper bound is only available when a feasible solution is found. Therefore, if the iterations are terminated earlier, an optimality gap cannot be obtained.
In this section, we introduce problem-specific heuristic methods
to enhance the performance of the Benders decomposition framework for the $k$-DefDom, by improving the bounds on the optimal objective value. In the first one, the goal is to generate an effective initial set of Benders feasibility cuts that could be included in the starting MP formulation. In the second heuristic method that we propose, we focus on finding an initial feasible defensive set which will constitute an initial upper bound on the objective value. To the best of our knowledge, this is the first attempt to solve the $k$-DefDom on general graphs in a heuristic way.

\subsection{A heuristic for initial cuts}
\label{section:heuristic1}

When the Benders master problem is initialized without any feasibility cuts, the resulting initial dual bounds are typically of low quality, as the model lacks information regarding the feasibility of candidate solutions. For this reason, it is standard practice to first generate a set of possibly high‑quality initial cuts. To this end, we propose a procedure that prioritizes attack sets with relatively small closed neighborhoods. In this method, vertices are first sorted according to increasing closed neighborhood size. Then, for each attack size from 1 to $k$, the vertices are considered as starting vertices in this order. Each attack is expanded by iteratively selecting the vertex that introduces the smallest number of additional vertices into the closed neighborhood of the current attack. Consequently, the generated attacks tend to have relatively small closed neighborhoods, and the collection of such attacks feeds the model with diverse information.

Algorithm~\ref{alg:initial-cuts} formalizes this procedure. The algorithm begins by sorting the vertices in increasing order of their closed neighborhood size $|N[v]|$ and determining the maximum number of cuts $q=\lceil\rho \cdot |V|\rceil$ to be generated for each attack size. Each vertex $s \in V$ is then used as a seed to construct attack sets of varying sizes $t \in {1, \dots, k}$ by greedily selecting additional vertices that minimize $|N[u]\setminus N[A]|$. For each generated attack $A$, a Hall-type constraint $\sum_{i \in N[A]} x_i \geq |A|$ is added to the master problem. This systematic procedure generates at most $k \times \lceil\rho \cdot |V|\rceil$ cut constraints, covering the full spectrum of attack sizes. By favoring vertices whose addition causes the smallest increase in the closed neighborhood of the attack, the procedure aims to construct attack sets for which $N[A]$ remains relatively small compared with $|A|$. These constraints are therefore typically tight and informative. By adding this diverse set of cuts before optimization begins, the heuristic strengthens the initial LP relaxation and provides a significantly better starting lower bound, potentially reducing the number of Benders iterations required to reach optimality.

\begin{algorithm}[ht]
\caption{Initial Cut Generation Heuristic for $k$-DefDom}
\label{alg:initial-cuts}
\small
\begin{algorithmic}[1]
\Require Graph $G=(V,E)$, parameter $k$, percentage $\rho$
\Ensure An initial set of Benders feasibility cuts
\State Sort vertices in increasing order of $|N[v]|$
\State Set $q\gets\lceil \rho \cdot|V|\rceil$
\For{$t=1,\dots,k$}
    \For{each vertex $s\in V$ until $q$ cuts are generated}
        \State Initialize $A\gets\{s\}$
        \While{$|A|<t$}
            \State Add a vertex $u\notin A$ minimizing $|N[u]\setminus N[A]|$
        \EndWhile
        \State Add the cut $\displaystyle\sum_{i\in N[A]}x_i\geq t$
    \EndFor
\EndFor
\end{algorithmic}
\end{algorithm}

\subsection{A clique-cover-based heuristic}
\label{section:warmstart}

Theorem \ref{thm:coNPhard} highlights the inherent computational difficulty of finding a solution to the $k$-DefDom, since any candidate solution must be validated as a feasible defensive set. In particular, for large values of $k$, this verification step becomes computationally prohibitive, as it requires checking whether every possible $k$-attack can be successfully countered. For this reason, the development of heuristic procedures that guarantee the feasibility of the obtained solutions represents a highly valuable research direction. In this section, we make use of clique covers of a graph to obtain a heuristic defensive set, ensuring that any $k$-attack can be countered.

Suppose that we are given a clique cover $\mathcal{C}$ of graph $G$. For each clique $C \in \mathcal{C}$, choosing $k$ arbitrary vertices of $C$ as defenders ensures that all attacks of size $k$ or smaller within $C$ can be defended, as every vertex in $C$ is adjacent to the defender vertices. Accordingly, if $k$
arbitrary vertices from every clique in $\mathcal{C}$ is chosen and added to the defender set $D$, then all possible $k$-attacks can be defended. If the size of a clique is less than $k$, then all vertices of that clique should be added to $D$ and the feasibility of $D$ is still guaranteed. In this approach, the size of the resulting set $D$ depends strongly on the size of the clique cover, thus the importance of finding a clique cover with as few cliques as possible. Besides, it is known that determining a minimum clique cover is NP-hard in general \citep{GJ}. To find a clique cover in the input graph, we solve the graph coloring problem in the complement graph and use the fact that each color class in the complement corresponds to a clique in the original graph. To this end, we adopt the Degree Saturation (DSATUR) heuristic \citep{dsatur}. As for the chordal graph instances (see Section \ref{sec:tests}), we find an optimal clique cover in linear time using the Perfect Elimination Ordering (PEO) \citep{chordalCC}. 

In the baseline version of this primal heuristic approach, we propose to construct the defender set by selecting exactly $\min(k,|C|)$ vertices for each clique $C$, and adding them to $D$. Although it ensures a feasible defensive set, this approach often yields a set of unnecessarily high cardinality. To reduce the size of $D$ while maintaining feasibility, we extend this method by a \textit{matching-based greedy reduction} mechanism, which exploits previously selected defenders to reduce redundant selections across cliques. We say that $v\in V$ is \emph{covered} by the set $D\subseteq V$ if $|N[v]\cap D|\ge k$, and \emph{uncovered} otherwise. A covered vertex can always be defended, since at most $k-1$ other vertices are attacked simultaneously. Consequently, a defender that is itself covered can be \emph{lent} to a vertex of another clique without losing its own protection. A defender can be lent only once, because two simultaneously attacked vertices need distinct defenders (e.g., the center of a star cannot defend two leaves at once), and the defenders already lent are collected in a set $T$.
Let $\mathcal{C}=\{C_1,C_2,\ldots, C_m\}$ be a clique cover of the graph. We initialize $D$ by selecting $\min(k,|C_1|)$ vertices of maximum degree from $C_1$. We then iterate through the remaining cliques. Let $U_i$ denote the set of uncovered vertices in $C_i$. For clique $C_i$, $i\geq 2$, we find a maximum matching $M$ (e.g. by Kuhn's Algorithm) between the covered defenders in $D$ that have not yet been lent and the uncovered vertices of $C_i$; this matching represents the assignment of existing defenders in $D$ to a subset of these vertices. We denote the set of vertices saturated by the matching $M$ by $V(M)$. The defenders used in the matching ($D\cap V(M)$) are said to be lent, and added to $T$. Then we update $U_i$ by removing its vertices saturated by the matching $M$ since they can be protected by the defenders in the current set $D$.
To guarantee that any $k$-attack at the vertices of $C_i$ can be countered, we select $\min(k, |U_i|)$ highest degree vertices of from $U_i$ and add them to $D$.

By accumulating vertices in $D$, subsequent cliques can potentially utilize these newly added defenders as lenders in their respective matchings, thereby minimizing the total size of the set with the help of overlaps. The algorithm terminates after all cliques in $\mathcal{C}$ have been processed, returning the final reduced defensive set $D$. The pseudo-code of the procedure is presented in Algorithm \ref{alg:warmstart} and an example is illustrated in Figure~\ref{fig:warmstart-example}. The feasibility of the returned set is proved in Proposition~\ref{prop:warmstart}.

\begin{algorithm}[ht]
\caption{Clique-Cover--Based Heuristic for $k$-DefDom}
\label{alg:warmstart}
\small
\begin{algorithmic}[1]
\Require Graph $G=(V,E)$, attack sets $\mathcal{A}$, parameter $k$, clique-cover $\mathsf{Method} \in \{\textsc{DSATUR}, \textsc{PEO}\}$
\Ensure A feasible defender set $D$
\State Compute a clique cover $\mathcal{C}=\{C_1,\ldots,C_m\}$ of $G$ using $\mathsf{Method}$
\State Sort cliques in $\mathcal{C}$ in descending order of cardinality
\State $D \gets \emptyset$; $T \gets \emptyset$ \Comment{$T$: defenders already lent} 
\For{$i$ from $1$ to $m$}
    \State $U_i \gets \{v \in C_i : |N[v]\cap D| < k\}$ \Comment{Vertices of $C_i$ with fewer than $k$ defenders in $N[v]$}
    \State Compute a maximum matching $M$ between $\{d \in D\setminus T : |N[d]\cap D|\ge k\}$ and $U_i$
    \State $T \gets T \cup (V(M)\cap D)$
    \State $U_i \gets U_i \setminus V(M)$ \Comment{Vertices in $U_i$ not saturated by $M$}
    \State $D \gets D \cup \{ \text{top } \min(k, |U_i|) \text{ vertices in } U_i \text{ by degree} \}$
\EndFor
\State \Return $D$
\end{algorithmic}
\end{algorithm}

\begin{figure}[ht]
\centering
\renewcommand{\kdcover}{\kdgray{a}{c}\kdgray{d}{e}\kdgray{b}{b}\kdgray{f}{f}}
\begin{subfigure}[t]{0.3\linewidth}\centering
  \kdpanel{}{}{}{}
  \caption{Initial state, $D=\emptyset$.}
\end{subfigure}\hspace{0.02\linewidth}
\begin{subfigure}[t]{0.3\linewidth}\centering
  \kdpanel{\kdband{a}{c}}{a,c}{}{}
  \caption{$C_1=\{a,c\}$: $a,c$ are added.}
\end{subfigure}\hspace{0.02\linewidth}
\begin{subfigure}[t]{0.3\linewidth}\centering
  \kdpanel{\kdband{d}{e}}{a,c,e}{a,c}{\draw[kdarr,bend left=40] (c) to (d);}
  \caption{$C_2=\{d,e\}$: $c$ lends to $d$; $e$ is added.}
\end{subfigure}

\vspace{0.8em}
\begin{subfigure}[t]{0.3\linewidth}\centering
  \kdpanel{\kdband{b}{b}}{a,b,c,e}{a}{}
  \caption{$C_3=\{b\}$: no adjacent lender; $b$ is added.}
\end{subfigure}\hspace{0.02\linewidth}
\begin{subfigure}[t]{0.3\linewidth}\centering
  \kdpanel{\kdband{f}{f}}{a,b,c,e,f}{a,b}{}
  \caption{$C_4=\{f\}$: no adjacent lender; $f$ is added.}
\end{subfigure}\hspace{0.02\linewidth}
\begin{subfigure}[t]{0.3\linewidth}\centering
  \kdpanel{}{a,b,c,e,f}{}{}
  \caption{Output $D=\{a,b,c,e,f\}$.}
\end{subfigure}
\caption{Execution of Algorithm~\ref{alg:warmstart} on the double star of Figure~\ref{fig:kdd_example} for $k=2$, with clique cover $C_1=\{a,c\}$, $C_2=\{d,e\}$, $C_3=\{b\}$, $C_4=\{f\}$. Gray regions show the cliques of the cover, the yellow region marks the clique being processed, blue vertices are defenders, green rings mark defenders that can lend that is, covered by the current $D$ but did not lend yet, and a green arrow shows a lending (the lender cannot lend again). The output has five defenders, whereas $\{a,c,d,e\}$ is optimal.}
\label{fig:warmstart-example}
\end{figure}

\begin{proposition}
\label{prop:warmstart}
Algorithm~\ref{alg:warmstart} returns a $k$-defensive dominating set $D$ in polynomial time.
\end{proposition}

\begin{proof}

Suppose not. By Hall's theorem, there is a subset $A$ with $|A|\le k$ and $|N[A]\cap D|<|A|$ .
We show that every $v\in A$ has a private defender $\pi(v)\in N[v]\cap D$ with $\pi$ injective, which contradicts $|N[A]\cap D|<|A|$.

Let $v\in A\cap C_i$ for some $i\in \{1,...,m\}$. Since $|N[v]\cap D|\le|N[A]\cap D|<k$ and $D$ only grows, $v$ is uncovered at iteration $i$, i.e., $v\in U_i$ in line~5.
If at least $k$ vertices are added to $D$ in line~9, all of them lie in $C_i$ and are adjacent to $v$, so $|N[v]\cap D|\ge k$, a contradiction.
Hence fewer than $k$ vertices remain in $U_i$ after line~8, and all of them are added to $D$. So either $v\in V(M)$ or $v\in D$.
Let $\pi(v)$ be the partner of $v$ in $M$ in the first case, and $v$ itself in the second.

\emph{Injectivity.} Partners are distinct within an iteration because $M$ is a matching, and across iterations because each partner enters $T$ in line~7 and is never matched again (line~6 uses $D\setminus T$).
A partner $d$ satisfies $|N[d]\cap D|\ge k$ when matched, whereas a self-defender $v\in A$ satisfies $|N[v]\cap D|<k$, so no partner equals a self-defender.
Self-defenders are distinct vertices. \emph{Running time.} A clique cover is computed in $O(N^2)$ time by DSATUR ($O(N+|E|)$ by PEO for chordal graphs). The algorithm performs $m\le N$ iterations, and the running time of each iteration is dominated by the maximum matching, which takes $O(N|E|)$ time. Hence the overall running time is $O(N^2|E|)$, which is polynomial.
\end{proof}

Algorithm \ref{alg:warmstart} is integrated into our Benders Decomposition framework to provide an initial feasible solution, and thus an initial upper bound. The performance of the heuristic algorithm was assessed in our experiments, where we observed that defender sets derived from a clique-cover can be substantially improved, i.e., reduced in size, by applying the matching-based procedure described above. This effect is particularly pronounced when a heuristic clique cover, rather than an optimal one, is employed. Recall that, there is no other known heuristic method for the $k$-DefDom for general graphs. Therefore, the only known (trivial) upper bound is the number of vertices, as it is guaranteed that $V$ is a feasible defensive set. The baseline version of the the clique-cover-based heuristic reduced the average initial upper bound from 200 to 135.83 for Erdős-Rényi graphs across all graph orders and densities. After applying the matching-based reduction, the average upper bound further decreases to 25.67, corresponding to an overall improvement of 81.1\%. Similarly, overall improvements of 91\% and 96\% were obtained for chordal and Barab\'asi--Albert graphs, respectively. The reader is referred to Appendix B Supplementary Material for the detailed results. We compare the quality of the solutions obtained by the clique-cover-based heuristic to the optimal values in Section \ref{section:IPvsBenders}.

\section{Algorithm Details}
\label{section:combinatorialCutGeneration}

The Benders decomposition approach can be implemented in two ways: (i) the traditional iterative manner where the MP is solved to optimality in each iteration, or (ii) within a branch-and-Benders (BB) scheme, where the MP is solved via a branch-and-bound method and the Benders cuts are generated dynamically at candidate (integer) solutions of the MP. For the latter BB approach \citep{gendron2014benders}, we integrate a lazy callback function into the MIP solver which allows us to process the solution at an integer node and add a cut if it is determined to be infeasible with respect to the original problem. The second approach is expected to improve the computational efficiency. 
Both methods are implemented and tested in this computational study, in addition to solving the IP model directly by the MIP solver, as a baseline method.

\paragraph{Solving the subproblems}
Let $\hat x \in \{0,1\}^{|V|}$ be the optimal solution of a B\&B subproblem. For generating Benders cuts violated at $\hat x$, we consider two fundamentally different approaches. In the first approach, dual subproblem(s) are solved to obtain extreme rays, either in an aggregate or disaggregated fashion. In other words, we either solve \DSP once or \DSPA for each attack $A$. In the former case, if \DSP is unbounded, the single aggregate cut that is generated using the dual extreme ray is added to the node subproblem to cut $\hat x$ off. 

In the latter disaggregated case, each \DSPA is solved sequentially until a violated cut -an attack that cannot be countered-- is found. Then, a single feasibility cut is added to the node subproblem.
If each \DSPA is feasible, no cuts are generated as $\hat x$ is a feasible solution. While the aggregate approach reduces model construction overhead and maintains all dual variables in memory, the decomposed approach offers better modularity and adherence to classical Benders theory at the cost of increased computational overhead from repeatedly constructing and solving multiple smaller dual models, with both methods ultimately converging to the optimal solution through different computational pathways.

The second approach exploits Theorem \ref{thm:combinatorialCut} to generate cuts in a combinatorial fashion, without solving dual subproblems. Instead, all subsets $S$ whose cardinality is at most $k$ are enumerated dynamically and checked for satisfying the first condition in Theorem \ref{thm:combinatorialCut}, i.e., the Hall violator condition of $|\{i\in N[S]: \hat{x}_i=1\}| < |S|$. The additional requirement of $N[j] \not \subseteq N[S]$ for all $j\in A \setminus S$ so that we get an extreme ray is implicitly handled during the cut selection procedure, which is explained in the following.

\paragraph{The multi-cut strategy}
For effective generation of combinatorial cuts, we develop a procedure employing (i) a cut-buffer with limited capacity ($C_{max}$) and (ii) a recursive depth-first-search enumeration of candidate subsets in order to sample violated cuts by exploiting the cut strength property defined in Proposition \ref{prop:strongerCut}. When searching for potential Hall violators to obtain a violated cut, we do not begin by exhaustively examining all singleton subsets. Instead, we expand the current set, regardless of whether it is a Hall violator, by considering its supersets. Once we arrive at a superset of size $k$, we backtrack and explore other supersets of size $k$, followed by other supersets of size $k-1$, and so forth. The procedure iterates over target subset sizes $k' = 1, \dots, k$ in increasing order. For each $k'$, vertices with $|N[v] \cap \{i \mid \hat{x}_i = 1\}| \geq k'$ are excluded from the search, as they cannot appear in a Hall violator of size $k'$, and the remaining candidates are visited in a random order for diversity. Subsets of size $k'$ are enumerated via recursive expansion, pruning any partial subset $S$ whenever $|N[S] \cap \{i \mid \hat{x}_i = 1\}| \geq k'$, since no superset of $S$ can be a Hall violator of size $k'$.

We terminate the search after exploring the maximum allowed number of candidate subsets ($B$) provided that at least one violated cut has been found. Otherwise, we continue until a violated cut is found or all subsets have been exhaustively evaluated.
Every time a subset $S\subseteq V$ giving a violated Benders cut is found during the search, the cut buffer is updated by exploiting the cut strength property defined in Proposition \ref{prop:strongerCut}, as detailed in Algorithm \ref{alg:budgeted-dfs-cuts}.
The main idea while updating the buffer is that no cut in the buffer is weaker than another. Therefore, we strictly prefer a superset $S'$ over a subset $S$ whenever the condition of Proposition \ref{prop:strongerCut} is satisfied. Furthermore, if the buffer fills up, we employ a maximum violation criterion to replace the least violated cuts, so that the master problem is fed only with the stronger or more violated constraints discovered during the search. 
As an example, choosing $B=1$ and $C_{max}=1$ corresponds to the special case of adding the first found violated cut. The values of these parameters might significantly change the performance of the algorithm and their impact is analyzed numerically in our experiments (See Appendix A Supplementary Material.)

\begin{algorithm}[htb]
\caption{Violated Cut Sampling}
\label{alg:budgeted-dfs-cuts}
\small
\begin{algorithmic}[1]
\Require Graph $G=(V,E)$, current solution $\hat{x}$, max size $k$, buffer limit $C_{max}$, attempt budget $B$
\Ensure A buffer $\mathcal{C}$ of strong violated cuts
\State $\mathcal{C} \leftarrow \emptyset$, $count \leftarrow 0$
\For{$k' = 1, \dots, k$}
    \State $V_{k'} \leftarrow \{v \in V \mid |N[v] \cap \{i \mid \hat{x}_i = 1\}| < k'\}$, randomly permuted \Comment{Node filtering and randomization}
    \State \Call{Expand}{$\emptyset,\, V_{k'},\, k'$}
\EndFor

\Procedure{Expand}{$S,\, U,\, k'$}
    \If{$count \ge B$ \textbf{and} $|\mathcal{C}| = C_{max}$} \Return \Comment{Early termination}\EndIf
    \State $count \leftarrow count + 1$
    \If{$|N[S] \cap \{i \mid \hat{x}_i = 1\}| \ge k'$} \Return \Comment{Pruning: no superset of $S$ can be a Hall violator of size $k'$}\EndIf
    \If{$|S| = k'$}
        \If{$|S| > |N[S] \cap \{i \mid \hat{x}_i = 1\}|$} \Call{UpdateBuffer}{$S$} \Comment{$S$ is a Hall violator}\EndIf
    \Else \Comment{Grow $S$ by adding vertices from $U$}
        \For{each $v \in U$ with $v > \max(S)$}
            \State \Call{Expand}{$S \cup \{v\},\, U,\, k'$}
        \EndFor
    \EndIf
\EndProcedure

\Procedure{UpdateBuffer}{$S_{new}$}
    \State $viol_{new} \leftarrow |S_{new}| - |N[S_{new}] \cap \{i \mid \hat{x}_i = 1\}|$ \Comment{$viol(S)=|S| - |N[S] \cap \{i \mid \hat{x}_i = 1\}|$}
    \For{$S_{old} \in \mathcal{C}$}
        \If{$S_{old} \subset S_{new}$ \textbf{and} $|S_{new}| - |S_{old}| \ge |N[S_{new}]| - |N[S_{old}]|$}
            \State $\mathcal{C} \leftarrow \mathcal{C} \setminus \{S_{old}\}$ \Comment{$S_{old}$ is discarded since $S_{new}$ gives a stronger cut}
        \ElsIf{$S_{new} \subset S_{old}$ \textbf{and} $|S_{old}| - |S_{new}| \ge |N[S_{old}]| - |N[S_{new}]|$}
            \State \Return \Comment{$S_{new}$ is discarded since a stronger cut is in the buffer}
        \EndIf
    \EndFor
    \If{$|\mathcal{C}| < C_{max}$}
        \State $\mathcal{C} \leftarrow \mathcal{C} \cup \{S_{new}\}$
    \Else \Comment{$S_{new}$ will replace an existing cut based on its violation}
        \State \textbf{let} $S_{min} \in \mathcal{C}$ be the cut with minimum violation
        \If{$viol_{new} > viol(S_{min})$}
            \State $\mathcal{C} \leftarrow (\mathcal{C} \setminus \{S_{min}\}) \cup \{S_{new}\}$
        \EndIf
    \EndIf
\EndProcedure
\end{algorithmic}
\end{algorithm}

\paragraph{Preprocessing of the attack scenario set}
Across all solution approaches, we reduce the number of attacks to be considered based on the following result shown in \cite{ekim2023defensive}: A set $D$ of vertices is $k$-defensive if and only if it defends against every $k$-attack $A$ such that $G^2[A]$ -the subgraph of $G^2$ induced by $A$- is connected, where $G^2$ denotes the square of a graph $G$ which is obtained from $G$ by adding an edge between all pairs of vertices of distance 2. Recalling that each defense ensures that defender nodes can be matched to all attacked nodes, we can express the same condition using Hall’s Theorem as follows: $D$ is $k$-defensive if and only if $|N[A] \cap D| \geq |A|$ for every set $A$ with at most $k$ vertices such that $G^2[A]$ is connected. 
Accordingly, we construct the scenario set $\mathcal{A}$ by explicitly enumerating all $\binom{\vert{}V\vert{}}{k}$ candidate size-$k$ subsets of $V$ and retaining only those for which $G^2[A]$ is connected.
In our experiments, we observe that for Erdős-Rényi graphs almost all attacks are connected in $G^2$. However, for chordal graph instances, considering only attacks which are connected in $G^2$ reduces the number of attacks in our formulation up to 83\% especially for sparse graphs, whereas this reduction is only 3\% for Barab\'asi--Albert graphs.

\section{Experimental Results}
\label{section:results}
In this section, we present the results of an extensive numerical study in which the performance of the proposed method is evaluated under several algorithmic settings. We first describe the instance generation procedure and the computational setup. Then, we present the experimental results and discuss our findings.

\subsection{Instance generation and test environment} \label{sec:tests}

We adopt three well established graph generation models while creating our test bed. These are Erdős-Rényi model for random graphs \citep{erdos1960evolution}, Barab\'asi-Albert model for scale-free graphs \citep{barabasi1999emergence}, and a PEO-based construction algorithm for chordal graphs, following the framework of \cite{cseker2022generation}. Despite the fact that chordal graphs are very structured, $k$-DefDom is NP-complete for fixed $k$ in chordal graphs (even in split graphs  which form a small subclass of chordal graphs \citep{ekim2020complexity}). Accordingly, the performance of our algorithms on chordal graphs as compared to other instance types, will inform us on the intrinsic difficulty of the problem with respect to the graph structure from an optimization point of view.

As the preliminary experiments indicated heavy dependence of the algorithm performance on the value of $k$ and on the graph type, as one would expect, we consider different graph orders for every value of $k$ we consider in the experiments. 
The number of vertices we consider ranges from 50 to 150 for Erdős-Rényi graphs, 100 to 1000 for Barab\'asi-Albert graphs, and 2500 to 3500 for chordal graphs. 
For Erdős-Rényi and chordal graphs, we consider three graph density levels, $p\in\{0.2,0.5,0.8\}$, representing sparse, moderate, and dense graphs, respectively. We exclude the case $p = 0.8$ for Barab\'asi–Albert instances, as this level of density cannot be attained under the standard preferential attachment model, unless the process is initialized with a large clique containing the majority of the vertices, which in turn produces a highly atypical and structurally extreme final graph.
For every $(N,p)$ pair where $N$ is the number of vertices and $p$ is the graph’s density, five random instances are generated using an appropriate model of the relevant graph class. 
While the $G(N,p)$ model was used for Erdős-Rényi graphs, the preferential attachment model is used for Barab\'asi–Albert graphs.
Chordal graph instances were generated using the publicly available implementation in \cite{cseker_chordal_github} (\url{https://github.com/oylumseker/Chordal/tree/master)}. 

To ensure a fair comparison, the same set of random instances was used when comparing different algorithmic settings. Different values of $k$ were tested while increasing the number of vertices, enabling the evaluation of the proposed methods across varying problem sizes and graph structures.
\

In addition to synthetic graphs, we consider
three real-world Istanbul security network instances obtained by setting different district sizes; this corresponds to different levels of maximum response time tolerated by the system. The small, medium, and large instances contain 50, 250, and 500 vertices and 70, 458, and 918 edges, respectively, corresponding to graph densities of 0.06, 0.02, and 0.01. All three instances are sparse, undirected, and connected. The computational results for these real-world instances are reported in Figure~\ref{fig:istanbul_real}.

The test instances are openly available to support reproducibility and further research at \cite{github} \url{(https://github.com/bilgeevrl/k-DefDom-Graph-Instances)}.

All algorithms were implemented in C++ with Gurobi 12.0 as the mixed-integer programming solver and tested on macOS using a MacBook with an M2 processor and 16 GB of RAM. To ensure a consistent computational setting, a time limit of 600 seconds per instance was used throughout all experiments. Gurobi was used at its default settings to solve the mixed-integer programming formulations, including the formulation (kDD) and the Benders Master Problem. One exception to the Gurobi settings is related to the use of the clique-cover-based heuristic.

Detailed results of the clique-cover-based heuristic are provided in Appendix B Supplemental Material.

\subsection{Comparison of all solution approaches}
\label{section:IPvsBenders}
The abbreviations of the algorithm variants used in the computational tables throughout the study are presented in Table \ref{tab:abbreviations}.

\begin{table}[!htbp]
\caption{Summary of abbreviations used in the computational results}
\label{tab:abbreviations}
\centering
\scriptsize
\resizebox{\textwidth}{!}{%
\begin{tabular}{lp{11cm}}
\toprule
\textbf{Abbreviation} & \textbf{Description} \\
\midrule
IP          & Integer programming formulation for $k$-DefDom (kDD) \\
Benders     & Classical Benders decomposition with aggregated cuts \\
BBA         & BB with aggregated cuts and LP-based subproblem (\DSP) \\
BBLP        & BB with decomposed cuts and LP-based subproblem (\DSPA) \\
BBC         & BB with decomposed cuts and combinatorial procedure (one cut per iteration) \\
BBMC        & BB with decomposed cuts and the combinatorial procedure that generates multiple cuts per iteration using a cut buffer (Algorithm~\ref{alg:budgeted-dfs-cuts})\\
BBMC+IC     & BBMC with initial cut generation heuristic (Algorithm \ref{alg:initial-cuts})\\
BBMC+H      & BBMC with the clique-cover-based heuristic (Algorithm~\ref{alg:warmstart}) providing an initial feasible solution \\
BBMC++      & BBMC with initial cuts and the clique-cover-based heuristic \\
\bottomrule
\end{tabular}%
}

\vspace{-0.2cm}
\end{table}

Figure~\ref{fig1} presents a performance comparison of all approaches from IP to BBMC++ on the same 75 Erd\H{o}s--R\'enyi instances, showing the number of optimally solved instances and the average optimality gap. Erd\H{o}s--R\'enyi graphs are used at this stage to develop and compare the algorithm variants, as they provide a controlled setting to evaluate the impact of each enhancement. As the methods progress from IP to BBMC++, the number of optimally solved instances increases from 15/75 to 63/75, and the average optimality gap decreases from 55.3\% to 5.3\%, corresponding to a 90.4\% improvement over the IP formulation. The results obtained on Erd\H{o}s--R\'enyi instances reveal a consistent and significant improvement in both solution quality and optimality gap as the methods advance from IP to the best performing variant BBMC++, which combines the initial cut generation heuristic and the clique-cover-based heuristic.

We first identify BBMC as the most promising approach with an in depth analysis of the results in Tables \ref{tab:comparison1} and \ref{tab:comparison2} for Erd\H{o}s--R\'enyi graphs. Then, we extend the experiments to chordal and Barab\'asi--Albert graph families to assess the performance of the heuristics. Graphics similar to Figure \ref{fig1} for chordal and Barab\'asi--Albert graphs (in Figure \ref{fig2}) will reveal that the improvement of the best performing variant BBMC++ over the IP formulation is even more striking for chordal and Barab\'asi--Albert graphs.

\begin{figure}[!htbp]
\centering
\begin{tikzpicture}[scale=1.15,transform shape]

\begin{axis}[
    name=bars,
    scale only axis,
    width=0.50\textwidth,
    height=4.8cm,
    xmin=-0.35, xmax=5.35,
    ymin=0, ymax=85,
    xtick={0,1,2,3,4,5},
    xticklabels={IP,Benders,BBA,BBLP,BBC,BBMC++},
    xticklabel style={
        rotate=20,
        anchor=north east,
        font=\sffamily\scriptsize
    },
    ytick={0,15,30,45,60,75},
    ylabel={Instances solved (out of 75)},
    ybar,
    bar width=15pt,
    axis x line*=bottom,
    axis y line*=left,
    ymajorgrids=true,
    grid style={gray!25},
    tick label style={font=\sffamily\scriptsize},
    label style={font=\sffamily\small},
    ylabel style={text=blue!70!black},
    yticklabel style={text=blue!70!black},
    axis line style={gray!55},
    enlarge x limits=false,
    clip=false
]

\addplot[
    draw=none,
    fill=blue!65!black
] coordinates {
    (0,15) (1,30) (2,35)
    (3,50) (4,55) (5,63)
};


\node[draw=blue!65!black,rounded corners=2pt,
      fill=white,inner sep=2pt,
      text=blue!65!black,
      font=\sffamily\tiny\bfseries,anchor=south]
      at (axis cs:0,15) [yshift=5pt] {15/75};

\node[draw=blue!65!black,rounded corners=2pt,
      fill=white,inner sep=2pt,
      text=blue!65!black,
      font=\sffamily\tiny\bfseries,anchor=south]
      at (axis cs:1,30) [yshift=5pt] {30/75};

\node[draw=blue!65!black,rounded corners=2pt,
      fill=white,inner sep=2pt,
      text=blue!65!black,
      font=\sffamily\tiny\bfseries,anchor=south]
      at (axis cs:2,35) [yshift=5pt] {35/75};

\node[draw=blue!65!black,rounded corners=2pt,
      fill=white,inner sep=2pt,
      text=blue!65!black,
      font=\sffamily\tiny\bfseries,anchor=south]
      at (axis cs:3,50) [yshift=5pt] {50/75};

\node[draw=blue!65!black,rounded corners=2pt,
      fill=white,inner sep=2pt,
      text=blue!65!black,
      font=\sffamily\tiny\bfseries,anchor=south]
      at (axis cs:4,55) [yshift=5pt] {55/75};

\node[draw=blue!65!black,rounded corners=2pt,
      fill=white,inner sep=2pt,
      text=blue!65!black,
      font=\sffamily\tiny\bfseries,anchor=south]
      at (axis cs:5,63) [yshift=5pt] {63/75};


\draw[fill=blue!65!black,draw=none]
    (rel axis cs:0.12,1.055)
    rectangle
    (rel axis cs:0.15,1.085);

\node[
    anchor=west,
    font=\sffamily\scriptsize
] at (rel axis cs:0.17,1.07)
    {Optimally solved};

\draw[
    orange!80!black,
    line width=1.3pt
] (rel axis cs:0.58,1.07)
  -- (rel axis cs:0.65,1.07);

\draw[
    orange!80!black,
    fill=white,
    line width=1pt
] (rel axis cs:0.615,1.07)
    circle (1.2pt);

\node[
    anchor=west,
    font=\sffamily\scriptsize
] at (rel axis cs:0.67,1.07)
    {Avg. gap};

\end{axis}

\begin{axis}[
    at={(bars.south west)},
    anchor=south west,
    scale only axis,
    width=0.50\textwidth,
    height=4.8cm,
    xmin=-0.35, xmax=5.35,
    ymin=0, ymax=70,
    xtick=\empty,
    ytick={0,10,20,30,40,50,60,70},
    ylabel={Avg. optimality gap (\%)},
    axis x line=none,
    axis y line*=right,
    tick label style={font=\sffamily\scriptsize},
    label style={font=\sffamily\small},
    ylabel style={text=orange!80!black},
    yticklabel style={text=orange!80!black},
    axis line style={gray!55},
    axis background/.style={fill=none},
    clip=false
]

\addplot[
    orange!80!black,
    line width=1.5pt,
    mark=o,
    mark size=2.5pt,
    mark options={
        fill=white,
        line width=1pt
    }
] coordinates {
    (0,55.3) (1,34.8) (2,29.4)
    (3,10.4) (4,8.0) (5,5.3)
};


\node[draw=orange!80!black,rounded corners=2pt,
      fill=white,inner sep=2pt,
      text=orange!80!black,
      font=\sffamily\tiny\bfseries,anchor=south]
      at (axis cs:0,55.3) [yshift=7pt] {55.3\%};

\node[draw=orange!80!black,rounded corners=2pt,
      fill=white,inner sep=2pt,
      text=orange!80!black,
      font=\sffamily\tiny\bfseries,anchor=south]
      at (axis cs:1,34.8) [yshift=7pt] {34.8\%};

\node[draw=orange!80!black,rounded corners=2pt,
      fill=white,inner sep=2pt,
      text=orange!80!black,
      font=\sffamily\tiny\bfseries,anchor=south]
      at (axis cs:2,8) [yshift=18pt] {29.4\%};

\node[draw=orange!80!black,rounded corners=2pt,
      fill=white,inner sep=2pt,
      text=orange!80!black,
      font=\sffamily\tiny\bfseries,anchor=south]
      at (axis cs:3,10.4) [yshift=7pt] {10.4\%};

\node[draw=orange!80!black,rounded corners=2pt,
      fill=white,inner sep=2pt,
      text=orange!80!black,
      font=\sffamily\tiny\bfseries,anchor=south]
      at (axis cs:4,8.0) [yshift=7pt] {8.0\%};

\node[draw=orange!80!black,rounded corners=2pt,
      fill=white,inner sep=2pt,
      text=orange!80!black,
      font=\sffamily\tiny\bfseries,anchor=south]
      at (axis cs:5,5.3) [yshift=7pt] {5.3\%};

\end{axis}

\end{tikzpicture}

\caption{Comparison of algorithm variants on Erdős-Rényi graphs}.
\label{fig1}

\end{figure}


Table~\ref{tab:comparison1} presents a detailed comparison of the performances the IP formulation (kDD), the classical (iterative) Benders Decomposition, and the BB approach where the cuts are added within the branch-and-bound tree using solver callbacks. At this stage of the experiments, the \textit{aggregate} dual subproblem \DSP is solved for cut generation. These exact methods are evaluated on Erd\H{o}s--R\'enyi graph instances with 15 configurations of attack size $k$, graph order $N$, and density $p$. Improvement represents the improvement over the first column of each table.

\begin{table}[htbp]
    \centering
    \caption{Comparison of IP Formulation, Classical Benders Decomposition, and Branch-and-Benders Cut}
    \label{tab:comparison1}
    \resizebox{\textwidth}{!}{
    \begin{tabular}{ccc|cccc|cccc|cccc}
        \toprule
        \multicolumn{3}{c|}{\textbf{Instance}} 
        & \multicolumn{4}{c|}{\textbf{IP}} 
        & \multicolumn{4}{c|}{\textbf{Benders}} 
        & \multicolumn{4}{c}{\textbf{BBA}} \\
        $k$ & $N$ & $p$ 
        & Opt. & Constr.(s) & Time(s) & Gap(\%) 
        & Opt. & Time(s) & Gap(\%) & \#Cuts 
        & Opt. & Time(s) & Gap(\%) & \#Cuts \\
        \bottomrule
        \multirow{3}{*}{2} & \multirow{3}{*}{50}  
        & 0.2 & 5/5 & 0.1 & 21.5 & 0.0  
        & 5/5 & 28.0 & 0.0 & 274  
        & 5/5 & 8.7 & 0.0 & 254 \\
        & & 0.5 & 5/5 & 0.2 & 123.5 & 0.0  
        & 5/5 & 43.3 & 0.0 & 194  
        & 5/5 & 31.8 & 0.0 & 195 \\
        & & 0.8 & 5/5 & 0.2 & 146.6 & 0.0  
        & 5/5 & 6.7 & 0.0 & 29   
        & 5/5 & 3.5 & 0.0 & 12 \\
        \midrule
        \multirow{3}{*}{2} & \multirow{3}{*}{100} 
        & 0.2 & 0/5 & 1.0 & -- & 48.8 
        & 0/5 & -- & 40.0 & 1665  
        & 0/5 & -- & 34.6 & 445 \\
        & & 0.5 & 0/5 & 1.6 & -- & 56.8 
        & 0/5 & -- & 34.5 & 1192  
        & 0/5 & -- & 25.0 & 76 \\
        & & 0.8 & 0/5 & 2.0 & -- & 20.0 
        & 5/5 & 580.2 & 0.0 & 213   
        & 5/5 & 514.1 & 0.0 & 52 \\
        \midrule
        \multirow{3}{*}{2} & \multirow{3}{*}{150} 
        & 0.2 & 0/5 & 3.9 & -- & 62.3 
        & 0/5 & -- & 50.0 & 2565  
        & 0/5 & -- & 31.8 & 46 \\
        & & 0.5 & 0/5 & 6.0 & -- & 98.0 
        & 0/5 & -- & 42.9 & 1099  
        & 0/5 & -- & 31.0 & 17 \\
        & & 0.8 & 0/5 & 7.5 & -- & 98.0 
        & 0/5 & -- & 25.0 & 760   
        & 5/5 & 580.1 & 0.0 & 7 \\
        \bottomrule
        \multirow{3}{*}{3} & \multirow{3}{*}{50}  
        & 0.2 & 0/5 & 2.1 & -- & 51.4 
        & 0/5 & -- & 30.0 & 738   
        & 0/5 & -- & 19.2 & 63 \\
        & & 0.5 & 0/5 & 3.7 & -- & 65.3 
        & 5/5 & 444.4 & 0.0 & 969   
        & 5/5 & 481.5 & 0.0 & 12 \\
        & & 0.8 & 0/5 & 4.8 & -- & 29.3 
        & 5/5 & 583.9 & 0.0 & 53    
        & 5/5 & 570.1 & 0.0 & 6 \\
        \midrule
        \multirow{3}{*}{3} & \multirow{3}{*}{100} 
        & 0.2 & 0/5 & 43.5 & -- & 100.0 
        & 0/5 & -- & 100.0 & 189   
        & 0/5 & -- & 100.0 & 51 \\
        & & 0.5 & 0/5 & 69.2 & -- & 100.0 
        & 0/5 & -- & 100.0 & 258   
        & 0/5 & -- & 100.0 & 10 \\
        & & 0.8 & 0/5 & 86.6 & -- & 100.0 
        & 0/5 & -- & 100.0 & 120   
        & 0/5 & -- & 100.0 & 10 \\
        \bottomrule
        \multicolumn{3}{c|}{\textbf{Total Opt.}}
        & \multicolumn{4}{c|}{\textbf{15/75}}
        & \multicolumn{4}{c|}{\textbf{30/75}}
        & \multicolumn{4}{c}{\textbf{35/75}} \\

        \multicolumn{3}{c|}{\textbf{Avg. Gap}}
        & \multicolumn{4}{c|}{\textbf{55.3}}
        & \multicolumn{4}{c|}{\textbf{34.8}}
        & \multicolumn{4}{c}{\textbf{29.4}} \\

        \multicolumn{3}{c|}{\textbf{Improvement}}
        & \multicolumn{4}{c|}{\textbf{--}}
        & \multicolumn{4}{c|}{\textbf{37.1\%}}
        & \multicolumn{4}{c}{\textbf{46.8\%}} \\
        \bottomrule
    \end{tabular}
    }
    \label{tab:comparison_final}
\vspace{0.3em}
\begin{minipage}{\textwidth}
\tiny
\textbf{Abbreviations:}
Constr. = Average construction time (sec);
Cuts = Average number of cuts added;
-- = No instance in the corresponding category was solved to optimality within the time limit; therefore, no average total time is reported.
Opt. = Number of instances solved to optimality out of 5;
Time = Average total time in seconds;
Gap = Average optimality gap (\%) computed over instances with positive gap only.
\end{minipage}
\end{table}

The results indicate a clear performance pattern. For smaller instances ($N=50$), all approaches solve most cases to optimality, although BBA is generally faster than both the IP formulation and classical Benders decomposition. However, as the instance size increases, the IP formulation shows clear scalability limitations, with optimality gaps growing substantially and reaching up to $100.0\%$. In addition, for many larger instances, no instance could be solved to optimality within the time limit. Classical Benders performs better than IP in several cases, especially for denser graphs, but BBA emerges as the strongest overall approach, attaining lower average total time and smaller gap, and solving several dense instances to optimality when the other methods fail to do so. This trend is also reflected in the summary, where both Benders decomposition based approaches improve upon the IP formulation in terms of average total runtime, with BBA yielding the best overall performance. It yields the smallest average optimality gap, \(29.4\%\), which corresponds to a \(46.8\%\) improvement over the IP formulation.

Table~\ref{tab:comparison2} compares four variants of the BB framework on the same set of instances used in the experiments of Table \ref{tab:comparison1}. For BBMC, we determine the number of attempts $B=50,000$ and the buffer size of $C_{\max}=50$ based on careful consideration and computational experiments. The computational results used for parameter selection are reported in Appendix A Supplemental Material. 

\FloatBarrier
\begin{table}[htbp]
\centering
\caption{Performance of the Branch-and-Benders-Cut algorithm under different cut generation mechanisms}
\label{tab:comparison2}
\resizebox{\textwidth}{!}{
\begin{tabular}{ccc|cccc|cccc|cccc|cccc}
\toprule
\multicolumn{3}{c|}{\textbf{Inst.}} &
\multicolumn{4}{c|}{\textbf{BBA}} &
\multicolumn{4}{c|}{\textbf{BBLP}} &
\multicolumn{4}{c|}{\textbf{BBC}} &
\multicolumn{4}{c}{\textbf{BBMC}} \\
$k$ & $N$ & $p$
& Opt. & Time & Gap & Cuts
& Opt. & Time & Gap & Cuts
& Opt. & Time & Gap & Cuts
& Opt. & Time & Gap & Cuts \\
\bottomrule
\multirow{3}{*}{2} & \multirow{3}{*}{50}
& 0.2 & 5/5 & 8.7 & 0.0 & 254 & 5/5 & 1.7 & 0.0 & 295 & 5/5 & 0.5 & 0.0 & 310 & 5/5 & 0.1 & 0.0 & 658 \\
& & 0.5 & 5/5 & 31.8 & 0.0 & 195 & 5/5 & 2.3 & 0.0 & 259 & 5/5 & 0.2 & 0.0 & 309 & 5/5 & 0.1 & 0.0 & 643 \\
& & 0.8 & 5/5 & 3.5 & 0.0 & 12 & 5/5 & 0.0 & 0.0 & 53 & 5/5 & 0.3 & 0.0 & 36 & 5/5 & 0.1 & 0.0 & 432 \\
\midrule
\multirow{3}{*}{2} & \multirow{3}{*}{100}
& 0.2 & 0/5 & -- & 34.6 & 445 & 5/5 & 289.1 & 0.0 & 1773 & 5/5 & 311.5 & 0.0 & 1817 & 5/5 & 254.4 & 0.0 & 1945 \\
& & 0.5 & 0/5 & -- & 25.0 & 76 & 5/5 & 78.3 & 0.0 & 1269 & 5/5 & 6.2 & 0.0 & 911 & 5/5 & 9.1 & 0.0 & 1536 \\
& & 0.8 & 5/5 & 514.1 & 0.0 & 52 & 5/5 & 36.6 & 0.0 & 517 & 5/5 & 0.3 & 0.0 & 164 & 5/5 & 0.3 & 0.0 & 863 \\
\midrule
\multirow{3}{*}{2} & \multirow{3}{*}{150}
& 0.2 & 0/5 & -- & 31.8 & 46 & 0/5 & -- & 30.8 & 3316 & 0/5 & -- & 35.7 & 3446 & 0/5 & -- & 29.6 & 3591 \\
& & 0.5 & 0/5 & -- & 31.0 & 17 & 0/5 & -- & 16.7 & 2399 & 0/5 & -- & 28.6 & 2608 & 1/5 & 499.5 & 16.7 & 3272 \\
& & 0.8 & 5/5 & 580.1 & 0.0 & 7 & 5/5 & 412.5 & 0.0 & 1664 & 5/5 & 8.1 & 0.0 & 1727 & 5/5 & 1.4 & 0.0 & 2204 \\
\bottomrule
\multirow{3}{*}{3} & \multirow{3}{*}{50}
& 0.2 & 0/5 & -- & 19.2 & 63 & 5/5 & 53.5 & 0.0 & 781 & 5/5 & 3.6 & 0.0 & 1261 & 5/5 & 4.7 & 0.0 & 2119 \\
& & 0.5 & 5/5 & 481.5 & 0.0 & 12 & 5/5 & 135.8 & 0.0 & 1204 & 5/5 & 2.6 & 0.0 & 940 & 5/5 & 1.8 & 0.0 & 1650 \\
& & 0.8 & 5/5 & 570.1 & 0.0 & 6 & 5/5 & 16.3 & 0.0 & 118 & 5/5 & 0.4 & 0.0 & 52 & 5/5 & 1.0 & 0.0 & 814 \\
\midrule
\multirow{3}{*}{3} & \multirow{3}{*}{100}
& 0.2 & 0/5 & -- & 100.0 & 51 & 0/5 & -- & 40.0 & 2205 & 0/5 & -- & 30.8 & 8777 & 0/5 & -- & 24.6 & 8841 \\
& & 0.5 & 0/5 & -- & 100.0 & 10 & 0/5 & -- & 28.6 & 1863 & 0/5 & -- & 25.0 & 9580 & 4/5 & 225.4 & 2.9 & 10634 \\
& & 0.8 & 0/5 & -- & 100.0 & 10 & 0/5 & -- & 40.0 & 397 & 5/5 & 22.5 & 0.0 & 666 & 5/5 & 3.4 & 0.0 & 1842 \\
\bottomrule
\multicolumn{3}{c|}{\textbf{Total Opt.}} 
& \multicolumn{4}{c|}{\textbf{35/75}} 
& \multicolumn{4}{c|}{\textbf{50/75}} 
& \multicolumn{4}{c|}{\textbf{55/75}} 
& \multicolumn{4}{c}{\textbf{60/75}} \\

\multicolumn{3}{c|}{\textbf{Avg. Gap}} 
& \multicolumn{4}{c|}{\textbf{29.4}} 
& \multicolumn{4}{c|}{\textbf{10.4}} 
& \multicolumn{4}{c|}{\textbf{8.0}} 
& \multicolumn{4}{c}{\textbf{4.9}} \\

\multicolumn{3}{c|}{\textbf{Improvement}}
& \multicolumn{4}{c|}{\textbf{--}} 
& \multicolumn{4}{c|}{\textbf{64.6\%}} 
& \multicolumn{4}{c|}{\textbf{72.8\%}} 
& \multicolumn{4}{c}{\textbf{83.3\%}} \\
\bottomrule
\end{tabular}}
\end{table}

The computational results show a clear performance difference among the four variants. For smaller instances ($N=50$), most methods solve the instances to optimality, but the combinatorial variants, BBC and BBMC, are generally faster than the LP-based variants BBA and BBLP. As the problem size increases, these differences become more visible. BBA appears to be the weakest approach overall, often failing to solve larger instances to optimality and leaving the largest optimality gaps. BBLP improves over BBA, especially on medium-sized instances, but it still has difficulty on harder cases such as $k=2$, $N=150$ and $k=3$, $N=100$.
Among the combinatorial variants, BBC already improves the results significantly by closing the gap in several cases where the LP-based methods cannot. However, BBMC gives the strongest overall performance. It generally achieves either the smallest optimality gap or the largest number of optimally solved instances, and in many of the difficult instance classes it is also the fastest approach. This can be seen clearly for $k=2$, $N=150$, $p=0.5$, where BBMC is the only method that solves some of the instances to optimality, and for $k=3$, $N=100$, where BBMC performs better than the other variants in terms of the final optimality gap. Although BBMC may generate more cuts than the other methods, generating cuts through the combinatorial subproblem is computationally more efficient than solving the LP-based subproblem, allowing a larger number of cuts to be added in a shorter amount of time. Overall, the results indicate that using the combinatorial procedure to generate cuts instead of LP-based subproblems strengthens the framework, and that generating multiple cuts in this setting makes BBMC the most scalable variant among the compared approaches.

In the next step, we assess the impact of integrating the heuristics proposed in Section \ref{section:heuristics} into the current best setting BBMC, for initial cut generation and for a warm start across various instance types. We set $\rho=0.25$ for Algorithm~\ref{alg:initial-cuts} in all experiments in which this algorithm is used, based on preliminary tests. In Tables \ref{tab:er_selected}, \ref{tab:chordal_selected} and \ref{tab:ba_selected}, the quality of the generated initial cuts was assessed by solving the LP relaxation of the master problem with these cuts and recording the objective as an initial lower bound ($LB_0$). Since the lower bound is zero before the introduction of any initial cuts, this procedure provides a direct way to evaluate how much the added cuts tighten the relaxation. 

For the BBMC+H setting, the size of the heuristic defensive set determines the initial upper bound ($UB_0$). The average final best upper bounds found among all four settings are presented in the last column. For the instances that can be solved optimally by one of the settings, this bound is equal to the optimal objective value. 
The running times reported in these tables include the time spent for the heuristic methods, which is measured in fractions of a second in all experiments.

\FloatBarrier
\begin{table}[!h]
\caption{Comparison of algorithm variants on Erd\H{o}s--R\'enyi graphs.}
\label{tab:er_selected}
\centering
\resizebox{\textwidth}{!}{
\begin{tabular}{c@{\hspace{10pt}}c c|
                ccc|
                cccc|
                cccc|
                ccc|
                c}
\toprule
& &
& \multicolumn{3}{c|}{BBMC}
& \multicolumn{4}{c|}{BBMC+IC}
& \multicolumn{4}{c|}{BBMC+H}
& \multicolumn{3}{c|}{BBMC++}
& Obj. \\
$k$ & $N$ & $p$
& Opt. & Time & Gap
& Opt. & Time & Gap & $LB_0$
& Opt. & Time & Gap & $UB_0$
& Opt. & Time & Gap
& value \\
\midrule
\multirow{12}{*}{\centering $2$}
& \multirow{3}{*}{75}
& 0.2
  & 5/5 & 4.6 & 0.0
  & 5/5 & 3.2 & 0.0 & 5.0
  & 5/5 & 4.0 & 0.0 & 17.8
  & 5/5 & 5.5 & 0.0
  & 9.6 \\
& & 0.5
  & 5/5 & 1.0 & 0.0
  & 5/5 & 0.8 & 0.0 & 2.5
  & 5/5 & 1.2 & 0.0 & 8.6
  & 5/5 & 0.4 & 0.0
  & 5.0 \\
& & 0.8
  & 5/5 & 0.1 & 0.0
  & 5/5 & 0.0 & 0.0 & 2.0
  & 5/5 & 0.1 & 0.0 & 4.4
  & 5/5 & 0.0 & 0.0
  & 3.0 \\
\cmidrule{2-18}
& \multirow{3}{*}{100}
& 0.2
  & 5/5 & 254.4 & 0.0
  & 5/5 & 239.0 & 0.0 & 5.2
  & 5/5 & 309.4 & 0.0 & 18.6
  & 5/5 & 171.7 & 0.0
  & 10.6 \\
& & 0.5
  & 5/5 & 9.1 & 0.0
  & 5/5 & 10.6 & 0.0 & 2.6
  & 5/5 & 9.8 & 0.0 & 8.6
  & 5/5 & 4.7 & 0.0
  & 5.0 \\
& & 0.8
  & 5/5 & 0.3 & 0.0
  & 5/5 & 0.1 & 0.0 & 2.0
  & 5/5 & 0.1 & 0.0 & 4.6
  & 5/5 & 0.1 & 0.0
  & 3.0 \\
\cmidrule{2-18}
& \multirow{3}{*}{125}
& 0.2
  & 0/5 & -- & 19.9
  & 0/5 & -- & 20.0 & 5.3
  & 0/5 & -- & 18.6 & 21.0
  & 0/5 & -- & 20.0
  & 12.0 \\
& & 0.5
  & 5/5 & 63.5 & 0.0
  & 5/5 & 124.5 & 0.0 & 2.6
  & 5/5 & 94.9 & 0.0 & 9.0
  & 5/5 & 131.2 & 0.0
  & 5.0 \\
& & 0.8
  & 5/5 & 0.6 & 0.0
  & 5/5 & 0.2 & 0.0 & 2.0
  & 5/5 & 0.3 & 0.0 & 5.2
  & 5/5 & 0.4 & 0.0
  & 3.0 \\
\cmidrule{2-18}
& \multirow{3}{*}{150}
& 0.2
  & 0/5 & -- & 29.6
  & 0/5 & -- & 28.5 & 5.2
  & 0/5 & -- & 29.5 & 22.0
  & 0/5 & -- & 27.3
  & 12.4 \\
& & 0.5
  & 1/5 & 499.5 & 16.7
  & 3/5 & 357.9 & 25.0 & 2.6
  & 2/5 & 425.6 & 22.2 & 8.6
  & 4/5 & 361.2 & 16.7
  & 5.6 \\
& & 0.8
  & 5/5 & 1.4 & 0.0
  & 5/5 & 0.6 & 0.0 & 2.0
  & 5/5 & 1.2 & 0.0 & 5.2
  & 5/5 & 1.2 & 0.0
  & 3.0 \\
\midrule
\multirow{9}{*}{\centering $3$}
& \multirow{3}{*}{75}
& 0.2
  & 4/5 & 185.3 & 0.0
  & 5/5 & 117.0 & 0.0 & 5.8
  & 4/5 & 178.6 & 8.3 & 22.6
  & 5/5 & 247.1 & 0.0
  & 11.2 \\
& & 0.5
  & 5/5 & 14.2 & 0.0
  & 5/5 & 8.3 & 0.0 & 3.3
  & 5/5 & 10.6 & 0.0 & 11.0
  & 5/5 & 10.2 & 0.0
  & 6.0 \\
& & 0.8
  & 5/5 & 2.4 & 0.0
  & 5/5 & 0.1 & 0.0 & 3.0
  & 5/5 & 0.5 & 0.0 & 6.6
  & 5/5 & 0.1 & 0.0
  & 4.0 \\
\cmidrule{2-18}
& \multirow{3}{*}{100}
& 0.2
  & 0/5 & -- & 24.6
  & 0/5 & -- & 23.3 & 5.9
  & 0/5 & -- & 23.1 & 25.2
  & 0/5 & -- & 21.8
  & 12.8 \\
& & 0.5
  & 4/5 & 225.4 & 14.3
  & 4/5 & 193.6 & 14.3 & 3.3
  & 5/5 & 321.8 & 0.0 & 10.6
  & 4/5 & 217.5 & 14.3
  & 6.2 \\
& & 0.8
  & 5/5 & 2.7 & 0.0
  & 5/5 & 0.3 & 0.0 & 3.0
  & 5/5 & 0.8 & 0.0 & 6.8
  & 5/5 & 0.4 & 0.0
  & 4.0 \\
\cmidrule{2-18}
& \multirow{3}{*}{125}
& 0.2
  & 0/5 & -- & 38.6
  & 0/5 & -- & 31.4 & 6.1
  & 0/5 & -- & 35.5 & 27.4
  & 0/5 & -- & 35.0
  & 14.8 \\
& & 0.5
  & 0/5 & -- & 28.6
  & 0/5 & -- & 28.6 & 3.4
  & 0/5 & -- & 28.6 & 12.0
  & 0/5 & -- & 28.6
  & 7.0 \\
& & 0.8
  & 5/5 & 3.4 & 0.0
  & 5/5 & 2.9 & 0.0 & 3.0
  & 5/5 & 3.7 & 0.0 & 6.8
  & 5/5 & 0.7 & 0.0
  & 4.0 \\
\midrule
\multicolumn{3}{c|}{\textbf{Total Opt.}}
& \multicolumn{3}{c|}{\textbf{74/105}}
& \multicolumn{4}{c|}{\textbf{77/105}}
& \multicolumn{4}{c|}{\textbf{76/105}}
& \multicolumn{3}{c|}{\textbf{78/105}}
& \\
\multicolumn{3}{c|}{\textbf{Avg. Gap}}
& \multicolumn{3}{c|}{\textbf{8.2}}
& \multicolumn{4}{c|}{\textbf{8.1}}
& \multicolumn{4}{c|}{\textbf{7.9}}
& \multicolumn{3}{c|}{\textbf{7.8}}
& \\
\multicolumn{3}{c|}{\textbf{Improvement}}
& \multicolumn{3}{c|}{\textbf{--}}
& \multicolumn{4}{c|}{\textbf{0.7\%}}
& \multicolumn{4}{c|}{\textbf{3.7\%}}
& \multicolumn{3}{c|}{\textbf{5.0\%}}
& \\
\bottomrule
\end{tabular}
}
\end{table}

The results reported in Table~\ref{tab:er_selected} provide a comparison of four variants of the BB framework on Erd\H{o}s--R\'enyi graph instances with additional $k$, $N$, and $p$ combinations. The results show that all four variants perform well on relatively small or dense instances, where most instances are solved optimally with very small computational times. The combined variant BBMC++ gives the most robust overall performance, frequently attaining the best computational times on the more difficult instance classes and solving more instances to optimality than the other variants. We will observe that the BBMC++ is also the best configuration for chordal graph and Barabási–Albert graph instances. 

The computational results on chordal graph instances are reported in Table~\ref{tab:chordal_selected}. The only difference among the solution approaches appears in the clique-cover-based heuristic variant (Algorithm \ref{alg:warmstart}). The clique cover for the heuristic solution is obtained using a perfect elimination ordering (PEO) instead of DSATUR algorithm, to exploit the structural properties of chordal graphs.
%
\FloatBarrier

\begin{table}[!t]
\caption{Comparison of algorithm variants on chordal graphs.}
\label{tab:chordal_selected}
\centering
\resizebox{\textwidth}{!}{
\begin{tabular}{c@{\hspace{10pt}}c c|
                ccc|
                cccc|
                cccc|
                ccc|
                c}
\toprule
& &
& \multicolumn{3}{c|}{BBMC}
& \multicolumn{4}{c|}{BBMC+IC}
& \multicolumn{4}{c|}{BBMC+H}
& \multicolumn{3}{c|}{BBMC++}
& Obj. \\
$k$ & $N$ & $p$
& Opt. & Time & Gap
& Opt. & Time & Gap & $LB_0$
& Opt. & Time & Gap & $UB_0$
& Opt. & Time & Gap
& value \\
\midrule

\multirow{9}{*}{\centering $12$}
& \multirow{3}{*}{3500}
& 0.2
  & 5/5 & 236.9 & 0.0
  & 5/5 & 42.8 & 0.0 & 56.6
  & 4/5 & 198.8 & 50.0 & 100.2
  & 5/5 & 37.0 & 0.0
  & 61.0 \\
& & 0.5
  & 5/5 & 125.1 & 0.0
  & 5/5 & 7.3 & 0.0 & 26.4
  & 5/5 & 126.5 & 0.0 & 40.8
  & 5/5 & 5.5 & 0.0
  & 26.4 \\
& & 0.8
  & 5/5 & 112.1 & 0.0
  & 5/5 & 4.1 & 0.0 & 12.0
  & 5/5 & 103.8 & 0.0 & 12.0
  & 5/5 & 3.4 & 0.0
  & 12.0 \\
\cmidrule{2-18}

& \multirow{3}{*}{4000}
& 0.2
  & 1/5 & 298.6 & 98.5
  & 5/5 & 97.5 & 0.0 & 66.8
  & 2/5 & 251.8 & 61.3 & 151.4
  & 5/5 & 95.6 & 0.0
  & 70.2 \\
& & 0.5
  & 5/5 & 159.1 & 0.0
  & 5/5 & 24.3 & 0.0 & 26.7
  & 5/5 & 142.0 & 0.0 & 65.2
  & 5/5 & 18.8 & 0.0
  & 27.8 \\
& & 0.8
  & 5/5 & 117.4 & 0.0
  & 5/5 & 14.6 & 0.0 & 14.5
  & 5/5 & 86.1 & 0.0 & 21.2
  & 5/5 & 8.5 & 0.0
  & 15.0 \\
\cmidrule{2-18}

& \multirow{3}{*}{4500}
& 0.2
  & 2/5 & 428.5 & 98.9
  & 5/5 & 40.6 & 0.0 & 59.0
  & 3/5 & 332.5 & 54.7 & 98.4
  & 5/5 & 40.3 & 0.0
  & 60.0 \\
& & 0.5
  & 5/5 & 179.8 & 0.0
  & 5/5 & 26.3 & 0.0 & 21.5
  & 5/5 & 177.1 & 0.0 & 27.4
  & 5/5 & 20.5 & 0.0
  & 23.6 \\
& & 0.8
  & 5/5 & 133.6 & 0.0
  & 5/5 & 6.6 & 0.0 & 12.0
  & 5/5 & 131.8 & 0.0 & 12.0
  & 5/5 & 5.6 & 0.0
  & 12.0 \\

\midrule

\multirow{9}{*}{\centering $14$}
& \multirow{3}{*}{3000}
& 0.2
  & 5/5 & 263.9 & 0.0
  & 5/5 & 21.8 & 0.0 & 68.3
  & 5/5 & 264.7 & 0.0 & 115.2
  & 5/5 & 17.9 & 0.0
  & 71.8 \\
& & 0.5
  & 5/5 & 160.9 & 0.0
  & 5/5 & 8.5 & 0.0 & 28.4
  & 5/5 & 161.0 & 0.0 & 57.0
  & 5/5 & 4.6 & 0.0
  & 30.8 \\
& & 0.8
  & 5/5 & 164.6 & 0.0
  & 5/5 & 30.8 & 0.0 & 14.0
  & 5/5 & 128.4 & 0.0 & 14.0
  & 5/5 & 2.9 & 0.0
  & 14.0 \\
\cmidrule{2-18}

& \multirow{3}{*}{3500}
& 0.2
  & 4/5 & 292.7 & 98.4
  & 5/5 & 43.6 & 0.0 & 64.5
  & 4/5 & 300.5 & 34.5 & 118.2
  & 5/5 & 40.3 & 0.0
  & 70.4 \\
& & 0.5
  & 5/5 & 158.4 & 0.0
  & 5/5 & 14.3 & 0.0 & 28.2
  & 5/5 & 172.5 & 0.0 & 48.8
  & 5/5 & 10.8 & 0.0
  & 30.8 \\
& & 0.8
  & 5/5 & 148.7 & 0.0
  & 5/5 & 5.0 & 0.0 & 14.0
  & 5/5 & 139.8 & 0.0 & 14.0
  & 5/5 & 4.2 & 0.0
  & 14.0 \\
\cmidrule{2-18}

& \multirow{3}{*}{4000}
& 0.2
  & 0/5 & -- & 98.5
  & 3/5 & 341.8 & 46.8 & 75.4
  & 0/5 & -- & 62.9 & 174.4
  & 3/5 & 316.5 & 48.6
  & 113.0 \\
& & 0.5
  & 5/5 & 372.8 & 0.0
  & 5/5 & 73.5 & 0.0 & 30.8
  & 3/5 & 225.7 & 39.9 & 74.0
  & 5/5 & 71.6 & 0.0
  & 31.6 \\
& & 0.8
  & 5/5 & 148.0 & 0.0
  & 5/5 & 17.0 & 0.0 & 17.4
  & 5/5 & 130.9 & 0.0 & 24.2
  & 5/5 & 12.6 & 0.0
  & 17.4 \\

\midrule

\multirow{9}{*}{\centering $16$}
& \multirow{3}{*}{2500}
& 0.2
  & 4/5 & 319.4 & 96.8
  & 5/5 & 29.7 & 0.0 & 69.6
  & 4/5 & 308.2 & 51.1 & 138.0
  & 5/5 & 45.2 & 0.0
  & 79.6 \\
& & 0.5
  & 5/5 & 229.5 & 0.0
  & 5/5 & 22.0 & 0.0 & 32.5
  & 5/5 & 190.8 & 0.0 & 54.4
  & 5/5 & 30.4 & 0.0
  & 32.0 \\
& & 0.8
  & 4/5 & 216.1 & 99.4
  & 5/5 & 2.8 & 0.0 & 16.0
  & 5/5 & 219.9 & 0.0 & 16.0
  & 5/5 & 2.2 & 0.0
  & 16.0 \\
\cmidrule{2-18}

& \multirow{3}{*}{3000}
& 0.2
  & 4/5 & 391.8 & 97.5
  & 5/5 & 56.3 & 0.0 & 82.4
  & 3/5 & 384.7 & 39.9 & 135.6
  & 5/5 & 61.6 & 0.0
  & 82.4 \\
& & 0.5
  & 5/5 & 197.1 & 0.0
  & 5/5 & 11.5 & 0.0 & 35.2
  & 5/5 & 189.6 & 0.0 & 64.8
  & 5/5 & 7.3 & 0.0
  & 35.2 \\
& & 0.8
  & 5/5 & 207.1 & 0.0
  & 5/5 & 23.6 & 0.0 & 16.0
  & 5/5 & 167.9 & 0.0 & 16.0
  & 5/5 & 3.4 & 0.0
  & 16.0 \\
\cmidrule{2-18}

& \multirow{3}{*}{3500}
& 0.2
  & 2/5 & 505.5 & 97.9
  & 5/5 & 127.8 & 0.0 & 73.6
  & 3/5 & 376.8 & 52.3 & 135.6
  & 5/5 & 117.8 & 0.0
  & 79.8 \\
& & 0.5
  & 5/5 & 195.5 & 0.0
  & 5/5 & 17.5 & 0.0 & 35.0
  & 5/5 & 207.6 & 0.0 & 48.2
  & 5/5 & 11.1 & 0.0
  & 35.2 \\
& & 0.8
  & 5/5 & 187.8 & 0.0
  & 5/5 & 6.1 & 0.0 & 16.0
  & 5/5 & 194.0 & 0.0 & 18.0
  & 5/5 & 5.3 & 0.0
  & 16.0 \\

\midrule
\multicolumn{3}{c|}{\textbf{Total Opt.}}
& \multicolumn{3}{c|}{\textbf{116/135}}
& \multicolumn{4}{c|}{\textbf{133/135}}
& \multicolumn{4}{c|}{\textbf{116/135}}
& \multicolumn{3}{c|}{\textbf{133/135}}
\\

\multicolumn{3}{c|}{\textbf{Avg. Gap}}
& \multicolumn{3}{c|}{\textbf{29.1}}
& \multicolumn{4}{c|}{\textbf{1.7}}
& \multicolumn{4}{c|}{\textbf{16.5}}
& \multicolumn{3}{c|}{\textbf{1.8}}
\\

\multicolumn{3}{c|}{\textbf{Improvement}}
& \multicolumn{3}{c|}{\textbf{--}}
& \multicolumn{4}{c|}{\textbf{94.0\%}}
& \multicolumn{4}{c|}{\textbf{43.2\%}}
& \multicolumn{3}{c|}{\textbf{93.8\%}}
\\
\bottomrule
\end{tabular}
}
\end{table}


We observe that our methods can solve chordal graph instances which are much larger than Erd\H{o}s--R\'enyi graphs within the same time limit: up to $N=4500$ for $k=12$ and $N=3500$ for $k=16$. This shows that, although the problem remains NP-complete for chordal graphs, they are easier from an optimization perspective. The results on chordal instances show that the difficulty varies considerably depending on $k$, $N$, and the density $p$. For smaller graph orders, all variants perform similarly and the differences are mainly in computational time. However, as the graph order and $k$ increase, the variants diverge significantly. For instance, at $k=12$, $N=4000$, and $p=0.2$, BBMC solves only one instance and BBMC+H solves two, whereas BBMC+IC and BBMC++ solve all five instances. Similarly, at $k=16$, $N=3000$, $p=0.2$, BBMC and BBMC+H leave gaps of 97.5\% and 39.9\%, respectively, while BBMC+IC and BBMC++ solve all instances to optimality. Among all variants, BBMC+IC and BBMC++ achieve the best overall performance, each solving 133 out of 135 instances, with average gaps of 1.7\% and 1.8\%, respectively, compared to 24.4\% for BBMC and 14.6\% for BBMC+H. Figure~\ref{fig2} presents a performance comparison of all variants on chordal graph instances on the left of the figure, highlighting the number of optimally solved instances and the average optimality gap across all 135 instances. The graphic shows a sharp contrast between IP and the classical Benders decomposition not being able to find even a feasible solution, and the best-performing variants, BBMC+IC and BBMC++, each solving 133 out of 135 instances to optimality within the given time limit.



The computational results on Barab\'asi--Albert graph instances are reported in Table~\ref{tab:ba_selected}. In this setting, BBMC+H denotes the variant equipped with a clique-cover-based heuristic obtained from DSATUR, and BBMC++ denotes the combined version using both MIS-based initial cuts and the DSATUR-based heuristic.
%
\begin{table}[!htbp]
\caption{Comparison of algorithm variants on Barab\'asi--Albert graphs.}
\label{tab:ba_selected}
\centering
\resizebox{\textwidth}{!}{
\begin{tabular}{c@{\hspace{10pt}}c c|
                ccc|
                cccc|
                cccc|
                ccc|
                c}
\toprule
& &
& \multicolumn{3}{c|}{BBMC}
& \multicolumn{4}{c|}{BBMC+IC}
& \multicolumn{4}{c|}{BBMC+H}
& \multicolumn{3}{c|}{BBMC++}
& Obj. \\
$k$ & $N$ & $p$
& Opt. & Time & Gap
& Opt. & Time & Gap & $LB_0$
& Opt. & Time & Gap & $UB_0$
& Opt. & Time & Gap
& value \\
\midrule
\multirow{6}{*}{\centering $2$}
& \multirow{2}{*}{250}
& 0.2
  & 1/5 & 67.3 & 14.5
  & 1/5 & 168.0 & 17.0 & 5.1
  & 1/5 & 160.7 & 9.1 & 30.6
  & 2/5 & 131.8 & 12.1
  & 11.2 \\
& & 0.5
  & 5/5 & 1.4 & 0.0
  & 5/5 & 0.5 & 0.0 & 2.0
  & 5/5 & 0.5 & 0.0 & 7.2
  & 5/5 & 0.5 & 0.0
  & 3.0 \\
\cmidrule{2-18}
& \multirow{2}{*}{500}
& 0.2
  & 0/5 & -- & 47.5
  & 0/5 & -- & 44.8 & 5.2
  & 0/5 & -- & 38.2 & 38.8
  & 0/5 & -- & 47.0
  & 15.2 \\
& & 0.5
  & 5/5 & 9.7 & 0.0
  & 5/5 & 8.9 & 0.0 & 2.0
  & 5/5 & 6.9 & 0.0 & 9.0
  & 5/5 & 9.7 & 0.0
  & 3.8 \\
\cmidrule{2-18}
& \multirow{2}{*}{1000}
& 0.2
  & 0/5 & -- & 99.4
  & 0/5 & -- & 55.2 & 5.0
  & 0/5 & -- & 66.0 & 45.0
  & 0/5 & -- & 58.4
  & 15.4 \\
& & 0.5
  & 5/5 & 74.1 & 0.0
  & 5/5 & 70.8 & 0.0 & 2.0
  & 5/5 & 75.7 & 0.0 & 11.6
  & 5/5 & 58.4 & 0.0
  & 4.0 \\
\midrule
\multirow{6}{*}{\centering $3$}
& \multirow{2}{*}{100}
& 0.2
  & 5/5 & 7.4 & 0.0
  & 5/5 & 2.6 & 0.0 & 5.8
  & 5/5 & 6.2 & 0.0 & 32.4
  & 5/5 & 2.5 & 0.0
  & 9.8 \\
& & 0.5
  & 5/5 & 3.1 & 0.0
  & 5/5 & 0.3 & 0.0 & 3.0
  & 5/5 & 1.2 & 0.0 & 10.0
  & 5/5 & 0.2 & 0.0
  & 4.0 \\
\cmidrule{2-18}
& \multirow{2}{*}{200}
& 0.2
  & 0/5 & -- & 21.3
  & 0/5 & -- & 19.2 & 6.0
  & 2/5 & 431.7 & 20.9 & 38.8
  & 1/5 & 582.8 & 18.3
  & 12.8 \\
& & 0.5
  & 5/5 & 4.7 & 0.0
  & 5/5 & 1.1 & 0.0 & 3.0
  & 5/5 & 2.7 & 0.0 & 10.8
  & 5/5 & 0.9 & 0.0
  & 4.0 \\
\cmidrule{2-18}
& \multirow{2}{*}{300}
& 0.2
  & 0/5 & -- & 45.0
  & 0/5 & -- & 32.5 & 4.6
  & 0/5 & -- & 48.5 & 43.2
  & 0/5 & -- & 31.5
  & 13.2 \\
& & 0.5
  & 5/5 & 9.3 & 0.0
  & 5/5 & 5.9 & 0.0 & 3.0
  & 5/5 & 9.8 & 0.0 & 9.4
  & 5/5 & 7.5 & 0.0
  & 4.0 \\
\midrule
\multirow{6}{*}{\centering $4$}
& \multirow{2}{*}{100}
& 0.2
  & 5/5 & 26.9 & 0.0
  & 5/5 & 23.6 & 0.0 & 6.5
  & 5/5 & 25.9 & 0.0 & 39.0
  & 5/5 & 20.3 & 0.0
  & 11.2 \\
& & 0.5
  & 5/5 & 4.8 & 0.0
  & 5/5 & 1.1 & 0.0 & 4.0
  & 5/5 & 3.5 & 0.0 & 13.6
  & 5/5 & 0.6 & 0.0
  & 5.0 \\
\cmidrule{2-18}
& \multirow{2}{*}{125}
& 0.2
  & 5/5 & 387.4 & 0.0
  & 3/5 & 278.3 & 0.0 & 6.5
  & 4/5 & 380.7 & 15.4 & 40.0
  & 5/5 & 236.3 & 15.4
  & 12.2 \\
& & 0.5
  & 5/5 & 6.3 & 0.0
  & 5/5 & 1.6 & 0.0 & 4.0
  & 5/5 & 4.8 & 0.0 & 10.0
  & 5/5 & 0.9 & 0.0
  & 5.0 \\
\cmidrule{2-18}
& \multirow{2}{*}{150}
& 0.2
  & 0/5 & -- & 8.3
  & 0/5 & -- & 9.7 & 6.4
  & 1/5 & 74.9 & 8.3 & 42.4
  & 1/5 & 441.6 & 8.3
  & 12.0 \\
& & 0.5
  & 5/5 & 7.2 & 0.0
  & 5/5 & 2.0 & 0.0 & 4.0
  & 5/5 & 6.3 & 0.0 & 11.4
  & 5/5 & 2.0 & 0.0
  & 5.0 \\
\midrule
\multicolumn{3}{c|}{\textbf{Total Opt.}}
& \multicolumn{3}{c|}{\textbf{61/90}}
& \multicolumn{4}{c|}{\textbf{59/90}}
& \multicolumn{4}{c|}{\textbf{63/90}}
& \multicolumn{3}{c|}{\textbf{64/90}}
& \\
\multicolumn{3}{c|}{\textbf{Avg. Gap}}
& \multicolumn{3}{c|}{\textbf{13.1}}
& \multicolumn{4}{c|}{\textbf{9.9}}
& \multicolumn{4}{c|}{\textbf{11.5}}
& \multicolumn{3}{c|}{\textbf{10.6}}
& \\
\multicolumn{3}{c|}{\textbf{Improvement}}
& \multicolumn{3}{c|}{\textbf{--}}
& \multicolumn{4}{c|}{\textbf{24.4\%}}
& \multicolumn{4}{c|}{\textbf{12.5\%}}
& \multicolumn{3}{c|}{\textbf{19.1\%}}
& \\
\bottomrule
\end{tabular}
}
\end{table}


In Table \ref{tab:ba_selected}, we observe once again that we can solve Barab\'asi--Albert instances which are much larger than Erd\H{o}s--R\'enyi graphs: up to $N=1000$ for $k=2$, and $N=150$ for $k=4$. The results on Barab\'asi--Albert instances indicate that the problem is particularly challenging for sparse graphs with $p=0.2$ and larger values of $N$, where none of the variants achieves full coverage. For denser instances with $p=0.5$, all variants generally perform well across all $k$ values. Among the compared variants, BBMC++ solves the largest number of instances, solving 64 out of 90 instances with an average gap of 10.6\%, compared to 61 out of 90 instances and an average gap of 13.1\% for BBMC. BBMC+IC achieves the smallest average gap of 9.9\%, while solving 59 out of 90 instances. For instance, at $k=3$, $N=300$, $p=0.2$, BBMC++ attains the smallest final gap of 31.5\%, while BBMC leaves a gap of 45.0\%. However, the relative contribution of each enhancement varies by instance class, and the improvements are less pronounced compared to the other graph families. Figure~\ref{fig2} presents a performance comparison of all variants on Barab\'asi--Albert instances on the right of the figure, highlighting the number of optimally solved instances and the average optimality gap across all 90 instances. The performance improvement from IP to BBMC++ is much more significant than for Erd\H{o}s--R\'enyi graphs, similar to the improvement trend obtained for chordal graphs.

As for a comparison of the results of the clique-cover-based heuristic with optimal values, one can compare the columns $UB_0$ under BBMC+H in Tables \ref{tab:er_selected}, \ref{tab:chordal_selected}, and \ref{tab:ba_selected} for Erdős-Rényi, chordal and Barab\'asi--Albert graphs, respectively. In all cases, the values remain within a range of 3.53 times optimal values (sometimes being much closer, and sometimes even hitting the best values), independently from the instance sizes. 

\FloatBarrier
\begin{figure}[!htbp]
\centering

\tikzset{
  solvedbox/.style={draw=blue!65!black,rounded corners=2pt,
    fill=white,inner sep=1.0pt,text=blue!65!black,
    font=\sffamily\scriptsize\bfseries},
  zerobox/.style={draw=gray!60,rounded corners=2pt,
    fill=white,inner sep=1.0pt,text=gray!65,
    font=\sffamily\scriptsize\bfseries},
  gapbox/.style={draw=orange!80!black,rounded corners=2pt,
    fill=white,inner sep=1.0pt,text=orange!80!black,
    font=\sffamily\scriptsize\bfseries},
}

\begin{tikzpicture}
\node[anchor=center,font=\sffamily\small] at (0,0)
    {\textcolor{blue!65!black}{\rule{2.5mm}{2.5mm}}
     \quad Optimally solved (left axis)
     \hspace{0.6cm}
     \tikz[baseline=-0.6ex]{
       \draw[orange!80!black,line width=1.5pt] (0,1mm)--(7mm,1mm);
       \draw[orange!80!black,fill=white,line width=1pt]
         (3.5mm,1mm) circle (1.2mm);}
     \quad Avg. gap (right axis)};
\end{tikzpicture}

\vspace{0.7em}

\begin{subfigure}[t]{0.485\textwidth}
\centering
\begin{tikzpicture}
\begin{axis}[
    name=barsCH,
    scale only axis,
    width=0.82\linewidth,
    height=4.8cm,
    xmin=-0.5, xmax=5.5,
    ymin=0, ymax=150,
    xtick={0,1,2,3,4,5},
    xticklabels={IP,Benders,BBMC,BBMC+IC,BBMC+H,BBMC++},
    xticklabel style={rotate=25,anchor=north east,font=\sffamily\scriptsize},
    ytick={0,25,50,75,100,125,150},
    ylabel={Instances solved (out of 135)},
    ybar,
    bar width=9pt,
    axis x line*=bottom,
    axis y line*=left,
    ymajorgrids=true,
    grid style={gray!25},
    ylabel style={text=blue!65!black,font=\sffamily\small},
    yticklabel style={text=blue!65!black,font=\sffamily\scriptsize},
    tick label style={font=\sffamily\scriptsize},
    axis line style={gray!55},
    clip=false,
]
\addplot[draw=none,fill=blue!65!black] coordinates {
    (0,0) (1,0) (2,116) (3,133) (4,116) (5,133)
};

\node[zerobox,anchor=south,yshift=4pt] at (axis cs:0,0) {0/135};
\node[zerobox,anchor=south,yshift=4pt] at (axis cs:1,0) {0/135};
\node[solvedbox,anchor=south,yshift=4pt] at (axis cs:2,116) {116/135};
\node[solvedbox,anchor=south,yshift=4pt] at (axis cs:3,133) {133/135};
\node[solvedbox,anchor=south,yshift=4pt] at (axis cs:4,116) {116/135};
\node[solvedbox,anchor=south,yshift=4pt] at (axis cs:5,133) {133/135};
\end{axis}

\begin{axis}[
    at={(barsCH.south west)},anchor=south west,
    scale only axis,
    width=0.82\linewidth,
    height=4.8cm,
    xmin=-0.5, xmax=5.5,
    ymin=0, ymax=110,
    xtick=\empty,
    ytick={0,20,40,60,80,100},
    yticklabels={,,,,,},
    axis x line=none,
    axis y line*=right,
    axis line style={gray!55},
    axis background/.style={fill=none},
    clip=false,
]
\addplot[orange!80!black,line width=1.5pt,mark=o,mark size=2.7pt,
         mark options={fill=white,line width=1pt}]
    coordinates {(0,100.0) (1,100.0) (2,29.1) (3,1.7) (4,16.5) (5,1.8)};

\node[gapbox,anchor=south,yshift=5pt] at (axis cs:0,100.0) {100.0\%};
\node[gapbox,anchor=north,yshift=-5pt] at (axis cs:1,100.0) {100.0\%};
\node[gapbox,anchor=south,yshift=5pt] at (axis cs:2,29.1) {29.1\%};
\node[gapbox,anchor=south,yshift=6pt] at (axis cs:3,1.7) {1.7\%};
\node[gapbox,anchor=south,yshift=5pt] at (axis cs:4,16.5) {16.5\%};
\node[gapbox,anchor=south,yshift=6pt] at (axis cs:5,1.8) {1.8\%};
\end{axis}
\end{tikzpicture}
\caption{Chordal graphs}
\label{fig:chordal_comparison_side}
\end{subfigure}%
\hfill%
\begin{subfigure}[t]{0.485\linewidth}
\centering
\begin{tikzpicture}
\begin{axis}[
    name=barsBA,
    scale only axis,
    width=0.82\linewidth,
    height=4.8cm,
    xmin=-0.5, xmax=5.5,
    ymin=0, ymax=72,
    xtick={0,1,2,3,4,5},
    xticklabels={IP,Benders,BBMC,BBMC+IC,BBMC+H,BBMC++},
    xticklabel style={rotate=25,anchor=north east,font=\sffamily\scriptsize},
    ytick={0,10,20,30,40,50,60,72},
    yticklabels={,,,,,,,},
    ybar,
    bar width=9pt,
    axis x line*=bottom,
    axis y line*=left,
    ymajorgrids=true,
    grid style={gray!25},
    tick label style={font=\sffamily\scriptsize},
    axis line style={gray!55},
    clip=false,
]
\addplot[draw=none,fill=blue!65!black] coordinates {
    (0,0) (1,0) (2,61) (3,59) (4,63) (5,64)
};
\node[zerobox,anchor=south,yshift=4pt] at (axis cs:0,0) {0/90};
\node[zerobox,anchor=south,yshift=4pt] at (axis cs:1,0) {0/90};
\node[solvedbox,anchor=south,yshift=4pt] at (axis cs:2,61) {61/90};
\node[solvedbox,anchor=south,yshift=4pt] at (axis cs:3,59) {59/90};
\node[solvedbox,anchor=south,yshift=4pt] at (axis cs:4,63) {63/90};
\node[solvedbox,anchor=south,yshift=4pt] at (axis cs:5,64) {64/90};
\end{axis}

\begin{axis}[
    at={(barsBA.south west)},anchor=south west,
    scale only axis,
    width=0.82\linewidth,
    height=4.8cm,
    xmin=-0.5, xmax=5.5,
    ymin=0, ymax=110,
    xtick=\empty,
    ytick={0,20,40,60,80,100},
    ylabel={Avg. optimality gap (\%)},
    axis x line=none,
    axis y line*=right,
    tick label style={font=\sffamily\scriptsize},
    ylabel style={text=orange!80!black,font=\sffamily\small},
    yticklabel style={text=orange!80!black,font=\sffamily\scriptsize},
    axis line style={gray!55},
    axis background/.style={fill=none},
    clip=false,
]
\addplot[orange!80!black,line width=1.5pt,mark=o,mark size=2.7pt,
         mark options={fill=white,line width=1pt}]
    coordinates {(0,98.1) (1,95.0) (2,13.1) (3,9.9) (4,11.5) (5,10.6)};

\node[gapbox,anchor=south,yshift=5pt] at (axis cs:0,98.1) {98.1\%};
\node[gapbox,anchor=north,yshift=-5pt] at (axis cs:1,95.0) {95.0\%};
\node[gapbox,anchor=south,yshift=5pt] at (axis cs:2,13.1) {13.1\%};
\node[gapbox,anchor=south,yshift=18pt] at (axis cs:3,9.9) {9.9\%};
\node[gapbox,anchor=south,yshift=5pt] at (axis cs:4,11.5) {11.5\%};
\node[gapbox,anchor=south,yshift=18pt] at (axis cs:5,10.6) {10.6\%};
\end{axis}
\end{tikzpicture}
\caption{Barab\'asi--Albert graphs}
\label{fig:ba_comparison_side}
\end{subfigure}

\caption{Comparison of algorithm variants on chordal and Barab\'asi--Albert graphs.}
\label{fig2}
\end{figure}

Figure~\ref{fig:istanbul_real} summarizes the performance of the proposed approach on the Istanbul real-world instances under one-hour time limit per instance. The results show that the smaller instances are solved to optimality within a few seconds, whereas the computational difficulty increases substantially with both $N$ and $k$. In particular, all instances with $N=500$ reach the time limit, with optimality gaps increasing from 1.40\% at $k=2$ to 48.4\% at $k=10$. A similar pattern is observed for $N=250$, where the $k=10$ instance reaches the time limit with a gap of 29.19\%. Overall, these results indicate that the method remains effective for moderate-sized real-world instances, while larger and more demanding cases become significantly harder to solve to optimality.

\begin{figure}[!htbp]
\centering
\begin{tikzpicture}[
    font=\sffamily,
    card/.style={
        draw=black!35,
        line width=0.4pt,
        rounded corners=2pt,
        minimum width=2.90cm,
        minimum height=1.38cm,
        text width=2.60cm,
        align=center,
        inner sep=3pt
    },
    opt/.style={
        card,
        fill=blue!3
    },
    med/.style={
        card,
        fill=orange!4
    },
    tl/.style={
        card,
        fill=red!4
    },
    colhead/.style={
        font=\bfseries\small
    },
    rowhead/.style={
        font=\bfseries\small
    }
]

\node[colhead] at (2.45,0) {$N=50$};
\node[colhead] at (5.55,0) {$N=250$};
\node[colhead] at (8.65,0) {$N=500$};

\node[rowhead,anchor=east] at (0.65,-1.15) {$k=2$};
\node[rowhead,anchor=east] at (0.65,-2.85) {$k=5$};
\node[rowhead,anchor=east] at (0.65,-4.55) {$k=10$};

\node[opt] at (2.45,-1.15) {
    {\scriptsize\textbf{Gap:} 0.0\%}\\[1pt]
    {\scriptsize\textbf{Time:} 0.84 sec.}\\[1pt]
    {\scriptsize\textbf{Obj. value:} 25}
};

\node[opt] at (5.55,-1.15) {
    {\scriptsize\textbf{Gap:} 0.0\%}\\[1pt]
    {\scriptsize\textbf{Time:} 3.97 sec.}\\[1pt]
    {\scriptsize\textbf{Obj. value:} 97}
};

\node[tl] at (8.65,-1.15) {
    {\scriptsize\textbf{Gap:} 1.40\%}\\[1pt]
    {\scriptsize\textbf{Time:} 3600 sec. (TL)}\\[1pt]
    {\scriptsize\textbf{Obj. value:} 211}
};

\node[opt] at (2.45,-2.85) {
    {\scriptsize\textbf{Gap:} 0.0\%}\\[1pt]
    {\scriptsize\textbf{Time:} 1.21 sec.}\\[1pt]
    {\scriptsize\textbf{Obj. value:} 33}
};

\node[med] at (5.55,-2.85) {
    {\scriptsize\textbf{Gap:} 0.0\%}\\[1pt]
    {\scriptsize\textbf{Time:} 63.94 sec.}\\[1pt]
    {\scriptsize\textbf{Obj. value:} 139}
};

\node[tl] at (8.65,-2.85) {
    {\scriptsize\textbf{Gap:} 19.88\%}\\[1pt]
    {\scriptsize\textbf{Time:} 3600 sec. (TL)}\\[1pt]
    {\scriptsize\textbf{Obj. value:} 262}
};

\node[opt] at (2.45,-4.55) {
    {\scriptsize\textbf{Gap:} 0.0\%}\\[1pt]
    {\scriptsize\textbf{Time:} 5.89 sec.}\\[1pt]
    {\scriptsize\textbf{Obj. value:} 38}
};

\node[tl] at (5.55,-4.55) {
    {\scriptsize\textbf{Gap:} 29.19\%}\\[1pt]
    {\scriptsize\textbf{Time:} 3600 sec. (TL)}\\[1pt]
    {\scriptsize\textbf{Obj. value:} 148}
};

\node[tl] at (8.65,-4.55) {
    {\scriptsize\textbf{Gap:} 48.4\%}\\[1pt]
    {\scriptsize\textbf{Time:} 3600 sec. (TL)}\\[1pt]
    {\scriptsize\textbf{Obj. value:} 258}
};

\node[
    draw=black!35,
    fill=blue!3,
    minimum width=0.25cm,
    minimum height=0.18cm
] at (2.25,-5.75) {};
\node[anchor=west,font=\scriptsize]
    at (2.45,-5.75) {fast / optimal};

\node[
    draw=black!35,
    fill=orange!4,
    minimum width=0.25cm,
    minimum height=0.18cm
] at (4.85,-5.75) {};
\node[anchor=west,font=\scriptsize]
    at (5.05,-5.75) {moderate time};

\node[
    draw=black!35,
    fill=red!4,
    minimum width=0.25cm,
    minimum height=0.18cm
] at (7.35,-5.75) {};
\node[anchor=west,font=\scriptsize]
    at (7.55,-5.75) {time limit};

\node[
    align=center,
    font=\scriptsize,
    text=black!65
] at (5.55,-6.35)
{Obj. value denotes the optimal objective value for optimally solved instances,
and the best bound at the time limit otherwise.};

\end{tikzpicture}

\caption{Performance summary for the Istanbul security network instances.
Each box reports the optimality gap, total computational time, and objective information.
(TL) stands for time limit.}
\label{fig:istanbul_real}
\end{figure}

\section{Conclusion and future work}
\label{section:conclusion}

This study investigates the application of Benders decomposition to the $k$-DefDom, a computationally challenging combinatorial optimization problem with practical relevance in network security and emergency response management. The proposed methods address two fundamental limitations of classical approaches: poor scalability and the absence of feasible solution methods in the literature.

The Benders decomposition framework, enhanced with strategies such as a combinatorial cut separation procedure, a cut-strength-based sampling method for adding multiple cuts, and mechanisms for generating initial lower and upper bounds, achieved markedly better performance than both the IP formulation and the classical Benders decomposition method. Across the three random graph families, the proposed framework solves 275 out of 330 instances to optimality with the final configuration. It solves Erd\H{o}s--R\'enyi instances with up to $N=150$, chordal instances with up to $N=4500$, and Barab\'asi--Albert instances with up to $N=1000$.
The clique-cover-based heuristic is the first heuristic approach in the literature designed to construct a defensive dominating set for general graphs in polynomial time. It improves the naive upper bound of $N$ by as much as 99.2\%. The initial cut-generation heuristic enhances the lower bounds and, in some instances, solves the problem entirely without requiring additional branching. Overall, the combined framework performs among the best variants across each of the investigated graph families, namely Erd\H{o}s--R\'enyi, chordal, and Barab\'asi--Albert. That being said, we note that the network structure plays an important role in the performance of the developed approaches: while Erd\H{o}s--R\'enyi graphs appear as the hardest instances, graphs which commonly model several network problems are easier to solve from an optimization point of view. Besides, although still as difficult as the general instances from a theoretical perspective, the structural properties of chordal graphs make them easier instances.
The proposed framework is also evaluated on nine real-world network instances derived from the map of Istanbul. Five of these instances are solved to optimality, while the remaining four reach the one-hour time limit with optimality gaps ranging from 1.4\% to 48.4\%. These results demonstrate the applicability of the proposed methods beyond synthetic graph instances.

The main limitation of the Benders decomposition method presented in this work is that the feasibility of a master solution can only be verified after exhaustively examining all possible attacks to determine whether they can be successfully countered (i.e., whether the subproblem is feasible). One potential research direction is to reformulate the problem so that it does not require enumerating every possible attack, for instance by adopting a bilevel defender--attacker formulation.
Future research could investigate the design of exact or heuristic solution methods tailored to particular classes of graph structures. One could also consider developing algorithmic solution procedures for the multi-set defensive domination variant where a defender can protect multiple nodes simultaneously. One can expect that allowing integer variables for defenders would simplify the problem. Another extension of practical relevance would be to allow attacks requiring more than one defender to be countered, representing more severe attacks.

\section*{Acknowledgments}
Bilge Varol has been supported by T\"UB\.ITAK B\.IDEB 2211. This work has been supported by the European Commission's Horizon Europe Research and Innovation programme through the Marie Skłodowska-Curie Actions Staff Exchanges (MSCA-SE) under Grant Agreement no.101182819 (COVER: (C)ombinatorial (O)ptimization for (V)ersatile Applications to (E)merging u(R)ban Problems).

	\begin{figure}[ht]
	   \centering
	   \includegraphics[width=0.3\linewidth]{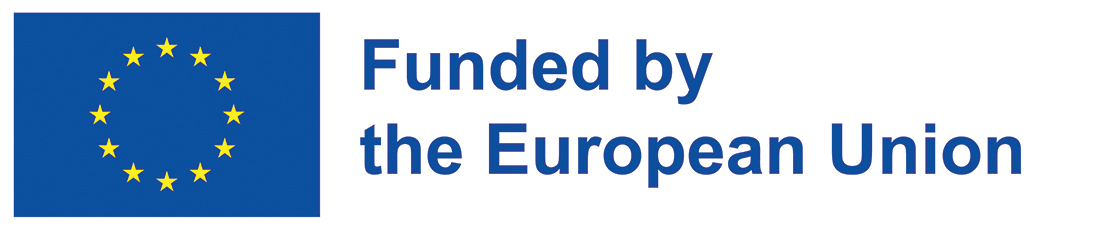}
    \end{figure}


\small

\bibliographystyle{elsarticle-harv} 

\bibliography{bib}

@book{haynes2013fundamentals,
  title={Fundamentals of domination in graphs},
  author={Haynes, Teresa W and Hedetniemi, Stephen and Slater, Peter},
  year={2013},
  publisher={CRC press}
}

@book{henning2014total,
  title={Total Domination in Graphs},
  author={Henning, M.A. and Yeo, A.},
  isbn={9781461465256},
  series={Springer Monographs in Mathematics},
  url={https://books.google.com.tr/books?id=BYe4BAAAQBAJ},
  year={2014},
  publisher={Springer New York}
}

@book{haynes2020topics,
  title={Topics in domination in graphs},
  author={Haynes, Teresa W and Hedetniemi, Stephen T and Henning, Michael A},
  volume={64},
  year={2020},
  publisher={Springer}
}

@book{haynes2023domination,
  title={Domination in graphs: Core concepts},
  author={Haynes, Teresa W and Hedetniemi, Stephen T and Henning, Michael A},
  year={2023},
  publisher={Springer}
}

@article{taskin2010benders,
  title={Benders decomposition},
  author={Task{\i}n, Z Caner},
  journal={Wiley Encyclopedia of Operations Research and Management Science. John Wiley \& Sons, Malden (MA)},
  year={2010}
}

@article{farley2004defensive,
  title={Defensive domination},
  author={Farley, Arthur M and Proskurowski, Andrzej},
  journal={Congressus Numerantium},
  volume={168},
  pages={97},
  year={2004},
  publisher={Winnipeg; Utilitas Mathematica; 1998}
}

@article{ekim2020complexity,
  title={The complexity of the defensive domination problem in special graph classes},
  author={Ekim, T{\i}naz and Farley, Arthur M and Proskurowski, Andrzej},
  journal={Discrete Mathematics},
  volume={343},
  number={2},
  pages={111665},
  year={2020},
  publisher={Elsevier}
}

@article{ekim2023defensive,
  title={Defensive domination in proper interval graphs},
  author={Ekim, T{\i}naz and Farley, Arthur and Proskurowski, Andrzej and Shalom, Mordechai},
  journal={Discrete Applied Mathematics},
  volume={331},
  pages={59--69},
  year={2023},
  publisher={Elsevier}
}

@article{gendron2014benders,
  title={Benders decomposition, branch-and-cut, and hybrid algorithms for the minimum connected dominating set problem},
  author={Gendron, Bernard and Lucena, Abilio and da Cunha, Alexandre Salles and Simonetti, Luidi},
  journal={INFORMS Journal on Computing},
  volume={26},
  number={4},
  pages={645--657},
  year={2014},
  publisher={INFORMS}
}

@article{harary1995conditional,
  title={Conditional graph-theory. 4. dominating sets},
  author={Harary, F and Haynes, TW},
  journal={Utilitas Mathematica},
  volume={48},
  pages={179--192},
  year={1995},
  publisher={UTIL MATH PUBL INC UNIV MANITOBA PO BOX 7 UNIV CENT, WINNIPEG, MANITOBA R3T~…}
}

@article{henning2025more,
  title={More on the complexity of defensive domination in graphs},
  author={Henning, Michael A and Pandey, Arti and Tripathi, Vikash},
  journal={Discrete Applied Mathematics},
  volume={362},
  pages={167--179},
  year={2025},
  publisher={Elsevier}
}

@inproceedings{henning2021approximation,
  title={Approximation Algorithm and Hardness Results for Defensive Domination in Graphs},
  author={Henning, Michael A and Pandey, Arti and Tripathi, Vikash},
  booktitle={International Conference on Combinatorial Optimization and Applications},
  pages={247--261},
  year={2021},
  organization={Springer}
}

@article{benders1962partitioning,
  title={Partitioning procedures for solving mixed-variables programming problems},
  author={Benders, J.F.},
  journal={Numerische Mathematik},
  volume={4},
  number={1},
  pages={238--252},
  year={1962}
}

@article{erdos1960evolution,
  title={On the evolution of random graphs},
  author={Erd\H{o}s, Paul and R{\'e}nyi, Alfr{\'e}d},
  journal={Publ. Math. Inst. Hungar. Acad. Sci},
  volume={5},
  pages={17--61},
  year={1960},
  publisher={Citeseer}
}

@article{DSATUR,
author = {Br\'{e}laz, Daniel},
title = {New methods to color the vertices of a graph},
year = {1979},
issue_date = {April 1979},
publisher = {Association for Computing Machinery},
address = {New York, NY, USA},
volume = {22},
number = {4},
issn = {0001-0782},
doi = {10.1145/359094.359101},
journal = {Commun. ACM},
month = apr,
pages = {251–256},
numpages = {6}
}

@article{chordalCC,
author = {Gavril, F\u{a}nic\u{a}},
title = {Algorithms for Minimum Coloring, Maximum Clique, Minimum Covering by Cliques, and Maximum Independent Set of a Chordal Graph},
journal = {SIAM Journal on Computing},
volume = {1},
number = {2},
pages = {180-187},
year = {1972},
doi = {10.1137/0201013}
}

@book{GJ,
author = {Garey, Michael R. and Johnson, David S.},
title = {Computers and Intractability; A Guide to the Theory of NP-Completeness},
year = {1990},
isbn = {0716710455},
publisher = {W. H. Freeman \& Co.},
address = {USA}
}

@article{dereniowski2019cops,
  title={Cops, a fast robber and defensive domination on interval graphs},
  author={Dereniowski, Dariusz and Gaven{\v{c}}iak, Tom{\'a}{\v{s}} and Kratochv{\'\i}l, Jan},
  journal={Theoretical Computer Science},
  volume={794},
  pages={47--58},
  year={2019},
  publisher={Elsevier}
}

@InProceedings{chaplick2025note,
  author =	{Chaplick, Steven and Gutowski, Grzegorz and Krawczyk, Tomasz},
  title =	{{A Note on the Complexity of Defensive Domination}},
  booktitle =	{50th International Symposium on Mathematical Foundations of Computer Science (MFCS 2025)},
  pages =	{35:1--35:15},
  series =	{Leibniz International Proceedings in Informatics (LIPIcs)},
  ISBN =	{978-3-95977-388-1},
  ISSN =	{1868-8969},
  year =	{2025},
  volume =	{345},
  editor =	{Gawrychowski, Pawe{\l} and Mazowiecki, Filip and Skrzypczak, Micha{\l}},
  publisher =	{Schloss Dagstuhl -- Leibniz-Zentrum f{\"u}r Informatik},
  address =	{Dagstuhl, Germany},
  doi =		{10.4230/LIPIcs.MFCS.2025.35}
}

@article{barabasi1999emergence,
  title={Emergence of scaling in random networks},
  author={Barab{\'a}si, Albert-L{\'a}szl{\'o} and Albert, R{\'e}ka},
  journal={Science},
  volume={286},
  number={5439},
  pages={509--512},
  year={1999},
  publisher={American Association for the Advancement of Science}
}

@article{cseker2022generation,
  title={Generation of random chordal graphs using subtrees of a tree},
  author={{\c{S}}eker, Oylum and Heggernes, Pinar and Ekim, Tinaz and Ta{\c{s}}k{\i}n, Z Caner},
  journal={RAIRO-Operations Research},
  volume={56},
  number={2},
  pages={565--582},
  year={2022},
  publisher={EDP Sciences}
}

@misc{cseker_chordal_github,
  author       = {Oylum {\c{S}}eker},
  title        = {Chordal: GitHub repository},
  year={2020},
  note         = {Available at: https://github.com/oylumseker/Chordal/tree/master}
}

@misc{github,
  author       = {Bilge Varol},
  title        = {GitHub repository of instances used in $k$-defensive domination problem},
  year={2026},
  note         = {Available at: https://github.com/bilgeevrl/k-DefDom-Graph-Instances}
}

@incollection{coversurveyreport,
  author      = {Chiarelli, Nina and Dallard, Cl{\'e}ment and Ekim, T{\i}naz and Furma{\'n}czyk, Hanna and Gollin, J. Pascal and Hilaire, Claire and Milani{\v{c}}, Martin and Mur{\v{s}}i{\v{c}}, Peter and Potgieter, Linke and Roux, Adriana},
  title       = {Chapter 9: The Armed Response Problem: Optimizing Private Security Patrols},
  booktitle   = {A Survey Report on Applications of Combinatorial Optimization Methods in Urban Planning and Relevant Mathematical Models},
  institution = {COVER Consortium},
  year        = {2025},
  month       = sep,
  type        = {Technical Report},
  number      = {Ref. Ares(2025)8244279},
  url         = {https://ec.europa.eu/research/participants/documents/downloadPublic?documentIds=080166e522ddfff5&appId=PPGMS},
  note        = {Project title: COVER (Combinatorial Optimization for Versatile Applications to Emerging uRban Problems). Funded under Marie Sk{\l}odowska-Curie Grant Agreement No: 101182819.}
}

@misc{continuousdefdom,
      title={Continuous Defensive Domination Problems}, 
      author={Christoph Grüne and Tom Janßen},
      year={2026},
      eprint={2605.10607},
      archivePrefix={arXiv},
      primaryClass={cs.CC},
      url={https://arxiv.org/abs/2605.10607}, 
}

\appendix
\section{The results of preliminary tests for parameter selection}
\label{section:Parameter}

\begin{table}[H]
\centering
\footnotesize
\caption{Parameter Optimization Results for Erd\H{o}s--R\'enyi Random Graphs using BBMC ($k=2, N=100$)}
\label{appendix}
\resizebox{\textwidth}{!}{
\begin{tabular}{cc|ccc|ccc|ccc|c}
\toprule
\multicolumn{2}{c|}{\textbf{Configuration}} &
\multicolumn{3}{c|}{\textbf{$p = 0.20$}} &
\multicolumn{3}{c|}{\textbf{$p = 0.50$}} &
\multicolumn{3}{c|}{\textbf{$p = 0.80$}} &
\textbf{Average} \\
$C_{\max}$ & $B$ &
Time & Cuts & Gap &
Time & Cuts & Gap &
Time & Cuts & Gap &
Time \\
\midrule

\multicolumn{11}{l}{\textit{Effect of $C_{\max}$ (with $B=50{,}000$ fixed)}} \\
\midrule
1   & 50,000
& 238.82 & 1946.0 & 0\%
& 9.21 & 1536.0 & 0\%
& 0.28 & 863.4 & 0\%
& 82.77 \\

50  & 50,000
& \textbf{197.64} & 1345.2 & 0\%
& \textbf{4.76} & 741.6 & 0\%
& \textbf{0.12} & 285.0 & 0\%
& \textbf{67.50} \\

100 & 50,000
& 253.91 & 2262.8 & 0\%
& 9.69 & 1912.0 & 0\%
& 0.42 & 1014.2 & 0\%
& 88.01 \\

\midrule
\multicolumn{11}{l}{\textit{Effect of $B$ (with $C_{\max}=50$ fixed)}} \\
\midrule
50 & 1,000
& 242.44 & 1945.3 & 0\%
& 9.52 & 1655.8 & 0\%
& 0.15 & 863.4 & 0\%
& 84.04 \\

50 & 50,000
& \textbf{197.64} & 1345.2 & 0\%
& \textbf{4.76} & 741.6 & 0\%
& \textbf{0.12} & 285.0 & 0\%
& \textbf{67.50} \\

50 & 100,000
& 241.95 & 1945.0 & 0\%
& 9.45 & 1536.0 & 0\%
& 0.28 & 863.4 & 0\%
& 83.89 \\

\midrule
\multicolumn{11}{l}{\textbf{Selected configuration:} $C_{\max}=50$, $B=50{,}000$} \\
\bottomrule
\end{tabular}}
\end{table}

\section{The results of the clique-cover-based heuristic}
\label{section:Appendix_heuristic}


\begin{table}[H]
\centering
\caption{Comparison of initial upper bounds produced by different clique-cover-based heuristic strategies across representative graph-order groups}
\label{tab:warmstart_result}
\resizebox{\textwidth}{!}{
\begin{threeparttable}
\begin{tabular}{ccc|cc|cccc|cc}
\toprule
\multirow{2}{*}{$p$} & \multirow{2}{*}{Size} & \multirow{2}{*}{$k$}
& \multicolumn{2}{c|}{ER}
& \multicolumn{4}{c|}{Chordal}
& \multicolumn{2}{c}{BA} \\
\cmidrule(lr){4-5}\cmidrule(lr){6-9}\cmidrule(lr){10-11}
& &
& DSATUR & DSATUR+M
& DSATUR & DSATUR+M & PEO & PEO+M
& DSATUR & DSATUR+M \\
\midrule

\multirow{12}{*}{0.2}
& \multirow{4}{*}{S} & 2  & 63 & 20 & 111 & 54 & 62 & 52 & 124 & 31 \\
&                   & 5  & 100 & 38 & 247 & 93 & 141 & 87 & 199 & 53 \\
&                   & 7  & 100 & 43 & 317 & 115 & 192 & 100 & 200 & 68 \\
&                   & 10 & 100 & 56 & 405 & 140 & 257 & 125 & 200 & 86 \\
\cmidrule(lr){2-11}

& \multirow{4}{*}{M} & 2  & 114 & 24 & 141 & 71 & 72 & 69 & 294 & 41 \\
&                   & 5  & 200 & 40 & 330 & 123 & 155 & 112 & 592 & 70 \\
&                   & 7  & 200 & 55 & 433 & 143 & 202 & 132 & 600 & 89 \\
&                   & 10 & 200 & 69 & 550 & 163 & 275 & 140 & 600 & 117 \\
\cmidrule(lr){2-11}

& \multirow{4}{*}{L} & 2  & 163 & 27 & 166 & 70 & 85 & 64 & 461 & 50 \\
&                   & 5  & 300 & 43 & 377 & 137 & 168 & 130 & 982 & 82 \\
&                   & 7  & 300 & 55 & 486 & 163 & 224 & 151 & 999 & 98 \\
&                   & 10 & 300 & 72 & 633 & 184 & 313 & 183 & 1000 & 126 \\

\midrule

\multirow{12}{*}{0.5}
& \multirow{4}{*}{S} & 2  & 36 & 10 & 65 & 21 & 31 & 30 & 113 & 12 \\
&                   & 5  & 86 & 14 & 151 & 49 & 67 & 55 & 158 & 14 \\
&                   & 7  & 100 & 21 & 195 & 67 & 91 & 70 & 174 & 23 \\
&                   & 10 & 100 & 27 & 256 & 66 & 118 & 75 & 186 & 29 \\
\cmidrule(lr){2-11}

& \multirow{4}{*}{M} & 2  & 62 & 11 & 85 & 34 & 31 & 32 & 323 & 7 \\
&                   & 5  & 152 & 18 & 196 & 62 & 65 & 57 & 406 & 14 \\
&                   & 7  & 195 & 25 & 261 & 63 & 92 & 75 & 455 & 20 \\
&                   & 10 & 200 & 31 & 344 & 72 & 135 & 83 & 511 & 26 \\
\cmidrule(lr){2-11}

& \multirow{4}{*}{L} & 2  & 84 & 11 & 105 & 53 & 39 & 47 & 525 & 8 \\
&                   & 5  & 208 & 18 & 243 & 83 & 84 & 70 & 643 & 16 \\
&                   & 7  & 278 & 23 & 320 & 92 & 119 & 81 & 717 & 17 \\
&                   & 10 & 300 & 31 & 417 & 118 & 162 & 102 & 813 & 25 \\

\midrule

\multirow{12}{*}{0.8}
& \multirow{4}{*}{S} & 2  & 17 & 4 & 45 & 15 & 12 & 11 & -- & -- \\
&                   & 5  & 42 & 10 & 107 & 27 & 27 & 20 & -- & -- \\
&                   & 7  & 58 & 12 & 143 & 34 & 37 & 23 & -- & -- \\
&                   & 10 & 80 & 22 & 192 & 45 & 50 & 26 & -- & -- \\
\cmidrule(lr){2-11}

& \multirow{4}{*}{M} & 2  & 26 & 5 & 60 & 25 & 14 & 22 & -- & -- \\
&                   & 5  & 65 & 10 & 137 & 36 & 40 & 35 & -- & -- \\
&                   & 7  & 91 & 14 & 184 & 42 & 55 & 39 & -- & -- \\
&                   & 10 & 130 & 17 & 249 & 46 & 79 & 50 & -- & -- \\
\cmidrule(lr){2-11}

& \multirow{4}{*}{L} & 2  & 38 & 6 & 72 & 25 & 19 & 19 & -- & -- \\
&                   & 5  & 92 & 10 & 157 & 42 & 41 & 31 & -- & -- \\
&                   & 7  & 128 & 14 & 207 & 49 & 55 & 39 & -- & -- \\
&                   & 10 & 182 & 18 & 277 & 54 & 77 & 51 & -- & -- \\

\bottomrule
\end{tabular}
\begin{tablenotes}[flushleft]
\footnotesize
\item Size groups are defined separately for each graph family. Erd\H{o}s--R\'enyi (ER) graphs use $(100,200,300)$ for $(S,M,L)$, chordal graphs use $(500,750,1000)$, and Barab\'asi--Albert (BA) graphs use $(200,600,1000)$. Initial upper-bound values ($UB_0$) are reported for each case, and the suffix +M denotes the complete version of the clique-cover-based heuristic, where a matching-based reduction is employed.
\end{tablenotes}
\end{threeparttable}
}
\end{table}




\end{document}